\documentclass[11pt,letterpaper,reqno]{amsart}

\usepackage{epsfig}
\usepackage{amsmath}
\usepackage{amssymb}
\usepackage{amsthm}
\usepackage{indentfirst}
\usepackage{xspace}
\usepackage{multirow}
\usepackage{xcolor}
\usepackage{verbatim}
\usepackage[letterpaper,margin=1in,headheight=15pt]{geometry}
\usepackage{mathpazo}
\usepackage{tikz-cd}
\usepackage{booktabs}
\usepackage{framed}
\usepackage{float}
\usepackage{enumitem}
\definecolor{shadecolor}{rgb}{0.85,0.85,0.85}
\setlist[itemize]{itemsep=1.5pt, topsep=4.5pt}

\newtheorem{thm}{Theorem}

\newtheorem{conj}[thm]{Conjecture}
\newtheorem{lem}[thm]{Lemma}
\newtheorem{prop}[thm]{Proposition}
\newtheorem{rem}[thm]{Remark}

\theoremstyle{definition}
\newtheorem{defn}[thm]{Definition}

\theoremstyle{definition}
\newtheorem{ex}[thm]{Example}

\numberwithin{thm}{section}
\numberwithin{equation}{section}

\newcommand{\fg}{{\mathfrak g}}
\newcommand{\fv}{{\mathfrak v}}
\newcommand{\fb}{{\mathfrak b}}

\newcommand{\cC}{\ensuremath{\mathcal C}}

\newcommand{\cB}{\ensuremath{\mathcal B}}
\newcommand{\cL}{\ensuremath{\mathcal L}}

\newcommand{\cO}{\ensuremath{\mathcal O}}

\newcommand{\cD}{\ensuremath{\mathcal D}}

\newcommand{\R}{\ensuremath{\mathbb R}}
\newcommand{\C}{\ensuremath{\mathbb C}}
\newcommand{\PP}{\ensuremath{\mathbb P}}
\newcommand{\Z}{\ensuremath{\mathbb Z}}

\newcommand{\half}{\ensuremath{\frac{1}{2}}}

\newcommand{\hk}{hyperk\"ahler\xspace}

\newcommand{\bu}{\ensuremath{\mathbf{u}}}
\newcommand{\bzero}{\ensuremath{\mathbf{0}}}

\newcommand{\I}{{\mathrm i}}
\newcommand{\e}{{\mathrm e}}
\newcommand{\de}{\mathrm{d}}

\newcommand{\abs}[1]{\lvert#1\rvert}
\newcommand{\norm}[1]{\lVert#1\rVert}
\newcommand{\IP}[1]{\langle#1\rangle}

\newcommand{\simarrow}{\xrightarrow\sim}

\newcommand{\ti}[1]{\textit{#1}}

\newcommand{\fsl}{\mathfrak{sl}}

\DeclareMathOperator{\ad}{ad}
\DeclareMathOperator{\Ad}{Ad}

\DeclareMathOperator{\Tr}{Tr}
\DeclareMathOperator{\End}{End}

\begin{document}

\title{Opers versus nonabelian Hodge}
\author[O. Dumitrescu]{Olivia Dumitrescu}
\author[L. Fredrickson]{Laura Fredrickson}
\author[G. Kydonakis]{Georgios Kydonakis}
\author[R. Mazzeo]{Rafe Mazzeo}
\author[M. Mulase]{Motohico Mulase}
\author[A. Neitzke]{Andrew Neitzke}

\date{}

{\abstract{For a complex simple simply connected Lie group $G$,
and a compact Riemann surface $C$,
we consider two sorts of families of flat $G$-connections over $C$.
Each family is determined by a point $\bu$ of the
base of Hitchin's integrable system for $(G,C)$. One family
$\nabla_{\hbar,\bu}$
consists of $G$-opers, and depends on
$\hbar \in \C^\times$. The other family
$\nabla_{R,\zeta,\bu}$
is built from solutions of
Hitchin's equations, and depends on
$\zeta \in \C^\times, R \in \R^+$.
We show that in the scaling limit
$R \to 0$, $\zeta = \hbar R$,
we have $\nabla_{R,\zeta,\bu} \to \nabla_{\hbar,\bu}$.
This establishes and generalizes a conjecture
formulated by Gaiotto.}}

\maketitle

\setcounter{page}{1}

\section{Introduction}

\subsection{Summary}

The main result of this paper is the proof of
an extension of a conjecture formulated by Gaiotto
in \cite{Gaiotto2014}, Conjecture \ref{conj:gaiotto} below.
This result concerns certain families of flat $G$-connections, for
a simple, simply connected complex Lie group $G$.
The case $G = SL(N,\C)$ is Theorem \ref{thm:oper-limit-SLN}
below. The general case is Theorem \ref{thm:oper-limit-G}.

\subsection{The case of {$G = SL(2,\C)$}}

Fix a compact Riemann surface $C$ of genus $g \ge 2$
and a holomorphic quadratic
differential $\phi_2$ on $C$. This data determines
two natural families of $SL(2,\C)$-connections
over $C$, as follows.

\medskip

\begin{itemize}
\item First, we consider the family of \ti{opers}
determined by $\phi_2$. These are global versions of
the locally-defined second-order differential operators
(Schr\"odinger operators)
\begin{equation} \label{eq:schrodinger-operator}
  \cD_{\hbar,\phi_2}: \psi(z) \mapsto \left[- \hbar^2 \partial_z^2 + P_2(z)\right] \psi(z),
\end{equation}
where $\hbar \in \C^\times$, and
\begin{equation}
\phi_2 = P_2(z) \, \de z^2
\end{equation}
locally.
The operator $\cD_{\hbar,\phi_2}$ makes sense globally with the following two stipulations:
\begin{itemize}
\item We consider $\psi(z)$ as a section of $K_C^{-1/2}$.
\item We use only coordinate charts in the atlas on $C$ coming from Fuchsian uniformization,
so that the transition maps are M\"obius transformations.
\end{itemize}
By a standard maneuver, replacing $\psi$ by its $1$-jet,
we can convert $\cD_{\hbar,\phi_2}$ to a
flat connection $\nabla_{\hbar,\phi_2}$ in a
rank $2$ vector bundle $E_\hbar$
over $C$. Holomorphically, $E_\hbar$ is an extension
\begin{equation}
  0 \to K_C^{1/2} \to E_\hbar \to K_C^{-1/2} \to 0.
\end{equation}
$E_\hbar$ has distinguished local trivializations defined canonically in terms of coordinate charts
on $C$, and in such a trivialization,
\begin{equation} \label{eq:oper-family-intro}
  \nabla_{\hbar,\phi_2} = \de + \hbar^{-1} \begin{pmatrix} 0 & P_2 \\ 1 & 0 \end{pmatrix} \de z, \qquad \hbar \in \C^\times.
\end{equation}

\item Second, we consider a \ti{Higgs bundle} determined by $\phi_2$:
this is the bundle
\begin{equation}
  E = K_C^{1/2} \oplus K_C^{-1/2}
\end{equation}
equipped with its standard holomorphic structure $\bar\partial_E$, and
a ``Higgs field'' $\varphi \in \Omega^{1,0}(\End E)$ represented
in local trivializations by
\begin{equation}
  \varphi = \begin{pmatrix} 0 &P_2 \\ 1 & 0 \end{pmatrix} \de z.
\end{equation}
According to the nonabelian Hodge theorem
(Theorem \ref{thm:harmonic-existence}), associated to
$(E,\bar\partial_E,\varphi)$
there is a canonical family of flat connections
in $E$, of the form
\begin{equation} \label{eq:hitchin-family-intro}
  \nabla_{\zeta,\phi_2} = \zeta^{-1} \varphi + D_h + \zeta \varphi^{\dagger_h}, \qquad \zeta \in \C^\times,
\end{equation}
where the Hermitian metric $h$ is determined by solving a certain elliptic PDE on $C$ (Hitchin's equation,
 \eqref{eq:hitchin-eq-unrescaled} below), and $D_h$ is the associated Chern connection.
\end{itemize}

\medskip

The families \eqref{eq:oper-family-intro},
\eqref{eq:hitchin-family-intro} are evidently similar;
in particular, their leading terms in the
$\hbar\to 0$ or $\zeta\to 0$ limit match, if we set $\hbar = \zeta$.
However, these two families are \ti{not} exactly the same.

Gaiotto in \cite{Gaiotto2014} proposed a relation between them,
as follows. Introduce an additional parameter $R \in \R^+$
and rescale the Higgs field by $\varphi \to R \varphi$; this leads to a 2-parameter analogue
of \eqref{eq:hitchin-family-intro},
\begin{equation} \label{eq:hitchin-family-intro-twoparam}
  \nabla_{\zeta,R,\phi_2} = R \zeta^{-1} \varphi + D_{h(R)} + R \zeta \varphi^{\dagger_{h(R)}}, \qquad \zeta \in \C^\times, R \in \R^+.
\end{equation}
Now fix $\hbar \in \C^\times$ and consider a scaling limit where
\begin{equation} \label{eq:scaling-limit}
 \zeta = R \hbar, \qquad R \searrow 0.
\end{equation}
(In other words, we take both $R \to 0$ and $\zeta \to 0$,
while holding their ratio $\zeta / R = \hbar$ fixed.)
Gaiotto conjectured that in this limit the connections
$\nabla_{\zeta,R,\phi_2}$ converge,
and that the limiting connection is an oper.
In \S\ref{sec:proof-SU2} below we prove that this is indeed
the case, and that the limiting oper
is equivalent to $\nabla_{\hbar,\phi_2}$ of \eqref{eq:oper-family-intro}.

\subsection{The case of {$G = SL(N,\C)$}}
The story just described has an extension where
we make the following replacements:
\begin{center}
\begin{tabular}{c|c}
quadratic differentials $\phi_2$ & tuples of holomorphic differentials $\bu = (\phi_2, \dots, \phi_N)$ \\
order-$2$ differential operators $\cD_{\hbar,\phi_2}$ & order-$N$ differential operators $\cD_{\hbar,\bu}$ \\
$SL(2,\C)$-connections & $SL(N,\C)$-connections
\end{tabular}
\end{center}
Both the families \eqref{eq:oper-family-intro} and
\eqref{eq:hitchin-family-intro-twoparam} admit generalizations to this setting.
As in the case $G = SL(2,\C)$, we show that these two families
are related by the scaling limit \eqref{eq:scaling-limit}.
This extension is Theorem \ref{thm:oper-limit-SLN}, proven
in \S\ref{sec:proof-SUN}. For the reader's convenience we also
review the construction of differential operators $\cD_{\hbar,\bu}$ generalizing \eqref{eq:schrodinger-operator},
in \S\ref{sec:SLN-opers-diff-ops}.

\subsection{The case of general {$G$}}
Finally we treat the case of a general simple, simply connected complex
Lie group $G$.
Once again, both families \eqref{eq:oper-family-intro} and
\eqref{eq:hitchin-family-intro-twoparam} admit generalizations to this setting,
and we show in Theorem \ref{thm:oper-limit-G} that these
two families are again related by the scaling limit \eqref{eq:scaling-limit}.

\subsection{Punctures}
Strictly speaking, the connections we consider are not quite the same as
those studied in \cite{Gaiotto2014}: that paper mainly
concerned \ti{meromorphic} quadratic
differentials $\phi_2$ on $\C\PP^1$,
rather than holomorphic ones on a Riemann surface $C$. We expect
that the methods of this paper can be generalized to the meromorphic
case, but we do not treat that extension here.

\subsection{Motivations}
One motivation for this work (as well as for \cite{Gaiotto2014})
is the desire to understand the relation
between asymptotics of flat sections of the two families of connections
\eqref{eq:oper-family-intro}, \eqref{eq:hitchin-family-intro}.
The analysis of the $\hbar \to 0$ asymptotic behavior of
Schr\"odinger equations, i.e.~of
\eqref{eq:oper-family-intro}, has a long history; it goes under
the name of the ``WKB approximation,'' more recently sharpened
to the ``exact WKB method.'' See e.g.
\cite{MR2182990,MR3280000} for highly readable accounts.
On the other hand, recently the $\zeta \to 0$ asymptotics
of the families \eqref{eq:hitchin-family-intro}
has been studied in \cite{Gaiotto:2009hg,Gaiotto2012}.
The two analyses address \ti{a priori} different problems,
and involve different methods. In \cite{Gaiotto:2009hg,Gaiotto2012}
the main analytic tool is a certain integral equation related to $tt^*$
geometry, while in the exact WKB method this role is played
by the Borel resummability of solutions of Schr\"odinger equations.
Nevertheless, the formal structures (Stokes graphs and connection
formulae) which appear in the two cases are the same.
Optimistically, the link
between \eqref{eq:oper-family-intro} and \eqref{eq:hitchin-family-intro}
provided by the results of \cite{Gaiotto2014} and
this paper might help in finding a direct
passage between these two asymptotic analyses.

A related motivation comes from physics: the result of this paper
should be helpful in understanding the relation between opers and
the Nekrasov-Shatashvili limit of the
Nekrasov partition function, as formulated most sharply
in \cite{Nekrasov:2011bc} (in the case $G = SL(2,\C)$.)

\subsection{Acknowledgements}

We thank the American Institute of Mathematics for its hospitality
during the workshop ``New perspectives on spectral data for Higgs bundles,''
where this work was initiated. We also thank the organizers and
all the participants of the workshop, and particularly Philip
Boalch for posing the question which led to this work.
We thank David Ben-Zvi and Lotte Hollands
for many useful conversations about opers.

This work is supported by National Science Foundation
awards DMS-1105050 (RM), DMS-1309298 (MM), DMS-1151693 (LF,AN).
OD is a member of the Simion Stoilow Institute of Mathematics of the Romanian Academy.

\section{Background, for {$G = SL(N,\C)$}} \label{sec:background-SLN}

In this section we give some background on the main players in our story:
Hitchin's equations, the Hitchin section, and opers. We specialize
to the case $G = SL(N,\C)$ and thus work with vector bundles rather than
principal bundles.

\subsection{Hitchin's equations} \label{sec:hitchin-equations-SLN}

Fix a compact Riemann surface $C$ of genus $g \ge 2$
and an integer $N \ge 2$. We consider tuples $(E, h, D, \varphi)$ comprised of:
\begin{itemize}
\item A rank $N$ complex vector bundle $E$ over $C$, equipped with a
trivialization of $\det E$,
\item A Hermitian metric $h$ in $E$
which induces the trivial metric on $\det E$,
\item An $h$-unitary connection $D$ in $E$,
\item A traceless section $\varphi$ of $\End(E) \otimes K_C$.
\end{itemize}
\ti{Hitchin's equations} \cite{MR89a:32021} are a system of nonlinear PDE
for these data:
\begin{subequations} \label{eq:hitchin-eq-unrescaled}
\begin{align}
 F_D + [\varphi, \varphi^{\dagger_h}] &= 0, \\
 \bar\partial_D \varphi &= 0.
\end{align}
\end{subequations}
Here $F_D$ denotes the curvature of $D$,
$\dagger_h$ means the adjoint with respect to the metric $h$,
and $\bar\partial_D$ is the $(0,1)$ part of the connection $D$.

We shall actually be considering a rescaled version of \eqref{eq:hitchin-eq-unrescaled},
\begin{subequations} \label{eq:hitchin-eq}
\begin{align}
 F_D + R^2 [\varphi, \varphi^{\dagger_h}] &= 0, \label{eq:hitchin-eq-1} \\
 \bar\partial_D \varphi &= 0, \label{eq:hitchin-eq-2}
\end{align}
\end{subequations}
obtained by replacing $\varphi \to R \varphi$, where $R \in \R^+$.

\subsection{Higgs bundles}

Now suppose given a solution $(E,h,D,\varphi)$ of \eqref{eq:hitchin-eq}.
The operator $\bar\partial_D$ gives a holomorphic structure on $E$.
Equation \eqref{eq:hitchin-eq-2} then says that $\varphi$ is a holomorphic section
of $\End(E) \otimes K_C$. Thus the tuple $(E, \bar\partial_D, \varphi)$ is
an $SL(N,\C)$-Higgs bundle:

\begin{defn}
An \ti{$SL(N,\C)$-Higgs bundle over $C$} is a tuple $(E, \bar\partial_E, \varphi)$:
\begin{itemize}
\item A rank $N$ complex vector bundle $E$ over $C$, equipped with a
trivialization of $\det E$,
\item A holomorphic structure $\bar\partial_E$ on $E$,
\item A traceless holomorphic section $\varphi$ of $\End(E) \otimes K_C$.
\end{itemize}
\end{defn}

\subsection{Harmonic metrics} \label{sec:harmonic-metrics}

Conversely, suppose given an $SL(N,\C)$-Higgs
bundle $(E, \bar\partial, \varphi)$ and
a Hermitian metric $h$ on $E$ inducing the trivial metric on $\det E$.
Then there is a unique $h$-unitary connection $D_h$ in $E$
whose $(0,1)$ part is
$\bar\partial_{D_h} = \bar\partial_E$ (Chern connection).
We write its $(1,0)$ part as $\partial_E^h$, and the full $D_h$ as
\begin{equation}
  D_h = \bar\partial_E + \partial_E^h.
\end{equation}
The equation \eqref{eq:hitchin-eq-2} automatically holds when
$D = D_h$.
The equation \eqref{eq:hitchin-eq-1} with $D = D_h$
becomes a nonlinear PDE for
the metric $h$:
\begin{defn}
Given an $SL(N,\C)$-Higgs bundle $(E, \bar\partial_E, \varphi)$, and $R \in \R^+$,
a \ti{harmonic metric with parameter $R$} is a Hermitian metric $h$ on $E$,
inducing the trivial metric on $\det E$, such that
\begin{equation} \label{eq:harmonic-metric}
F_{D_h} + R^2 [\varphi, \varphi^{\dagger_h}] = 0.
\end{equation}
\end{defn}

Thus, we have
\begin{rem} \label{remark:harmonic-hitchin}
Given an $SL(N,\C)$-Higgs bundle $(E, \bar\partial_E, \varphi)$,
$R \in \R^+$,
and a harmonic metric $h$ with parameter $R$,
the tuple $(E, h, D_h, \varphi)$ gives
a solution of Hitchin's equations \eqref{eq:hitchin-eq}.
\end{rem}

Next we consider the existence of harmonic metrics.
\begin{defn}
An $SL(N,\C)$-Higgs bundle $(E, \bar\partial_E, \varphi)$
is called \ti{stable} if there is no holomorphic subbundle
$E' \subset E$ such that $\varphi(E') \subset E' \otimes K_C$
and $\deg(E') > 0$.
\end{defn}

The following key result
(``nonabelian Hodge theorem'') is proven in
\cite{MR944577}:\footnote{More
precisely, the theorem in \cite{MR944577} concerns $GL(N,\C)$-bundles
rather than $SL(N,\C)$-bundles, but it is straightforward to deduce
the version for $SL(N,\C)$-bundles.}
\begin{thm} \label{thm:harmonic-existence}
Given a stable $SL(N,\C)$-Higgs bundle $(E, \bar\partial_E, \varphi)$,
and any $R \in \R^+$,
there exists a unique harmonic metric $h$ with parameter $R$.
\end{thm}

Combining this with Remark \ref{remark:harmonic-hitchin},
we see that given a stable Higgs bundle and a parameter $R$,
we obtain a solution of Hitchin's equations \eqref{eq:hitchin-eq}
with parameter $R$.

\subsection{Real twistor lines}

Given a solution $(E, h, D, \varphi)$ of
Hitchin's equations \eqref{eq:hitchin-eq} with parameter $R$,
there is a corresponding family of flat non-unitary connections
in $E$, given by the formula
\begin{equation} \label{eq:hitchin-family}
  \nabla_\zeta = \zeta^{-1} R \varphi + D + \zeta R \varphi^{\dagger_h}, \qquad \zeta \in \C^\times.
\end{equation}
Indeed, the statement that $\nabla_\zeta$ is flat for all $\zeta \in \C^\times$
is equivalent to \eqref{eq:hitchin-eq}.
The family \eqref{eq:hitchin-family} is sometimes called a ``real twistor
line,'' because of the role it plays in the twistorial description of the
\hk metric on the moduli space of solutions of \eqref{eq:hitchin-eq}.

\subsection{The principal {$\fsl(2,\C)$}-triple}

Define
{\allowdisplaybreaks
\begin{align}
 H &=
        \begin{pmatrix} N-1 & & & & \\ & N-3 & & & \\ & & \ddots & & \\
      & & & -N+3 & \\ & & & & -N+1
    \end{pmatrix}, \label{eq:H-SLN} \\
 X_+ &= \begin{pmatrix} 0 & \sqrt{r_1} & & & \\
              &  0 & \sqrt{r_2} & & \\
              & & \ddots & \ddots& \\
               & & & 0& \sqrt{r_{N-1}}\\
                & & & & 0
              \end{pmatrix}, \label{eq:Xplus-SLN} \\
 X_- &= \begin{pmatrix} 0 &  & & & \\
              \sqrt{r_1}&  0 &  & & \\
              & \sqrt{r_{2}}& \ddots & & \\
               & & \ddots & 0& \\
                & & & \sqrt{r_{N-1}}& 0
              \end{pmatrix}, \label{eq:Xminus-SLN}
\end{align}}
\!\!where
\begin{equation}
	r_i = i(N-i).
\end{equation}
These make up an $\fsl(2,\C)$-triple:
\begin{equation} \label{eq:SL2-relations}
 [H, X_\pm] = \pm 2 X_\pm, \qquad [X_+, X_-]=H.
 \end{equation}

In addition, for each $n \ge 1$, choose (once and for all) a nonzero matrix $X_n$, such that only
the $ij$ entries with $j-i=n$ are nonzero (the $n^{\mathrm{th}}$ superdiagonal), or equivalently
\begin{equation} \label{eq:Xn-grade}
	[H, X_n] = 2n X_n,
\end{equation}
and also
\begin{equation} \label{eq:Xplus-commutator}
	[X_+, X_n] = 0.
\end{equation}
For example, when $N=4$ we could choose
\begin{equation}
  X_1 = \begin{pmatrix} 0 & \sqrt{3} & 0 & 0 \\ 0 & 0 & 2 & 0 \\ 0 & 0 & 0 & \sqrt{3} \\ 0 & 0 & 0 & 0 \end{pmatrix}, \quad   X_2 = \begin{pmatrix} 0 & 0 & 1 & 0 \\ 0 & 0 & 0 & 1 \\ 0 & 0 & 0 & 0 \\ 0 & 0 & 0 & 0 \end{pmatrix}, \quad  X_3 = \begin{pmatrix} 0 & 0 & 0 & 1 \\ 0 & 0 & 0 & 0 \\ 0 & 0 & 0 & 0 \\ 0 & 0 & 0 & 0 \end{pmatrix}.
\end{equation}
For later use we record a few facts, obtained by direct computation:
\begin{prop} \label{prop:Xn-facts-SLN} We have the following:
\begin{itemize}
\item The equations \eqref{eq:Xn-grade}, \eqref{eq:Xplus-commutator}
determine $X_n$ up to a scalar multiple for $n > 0$,
and the solution $X_n$ has the antidiagonal symmetry $(X_n)_{ij} = (X_n)_{N+1-j,N+1-i}$.
\item The equations \eqref{eq:Xn-grade}, \eqref{eq:Xplus-commutator}
determine $X_0$ to be a multiple of the identity.
\item The equations \eqref{eq:Xn-grade}, \eqref{eq:Xplus-commutator}
have only the solution $X_n = 0$ for $n < 0$.
\end{itemize}
\end{prop}

\subsection{The Hitchin component} \label{sec:hitchin-component-SLN}

\begin{defn}
The \ti{Hitchin base} is the vector space
\begin{equation}
	 \cB = \bigoplus_{n=2}^{N} H^0(C, K_C^{n}).
\end{equation}
\end{defn}
We denote points of $\cB$ by
\begin{equation}
	\bu = (\phi_2, \dots, \phi_N).
\end{equation}

Now fix a spin structure on $C$, i.e.~a
holomorphic line bundle $\cL$ over $C$ equipped with an isomorphism
$\cL^2 \simeq K_C$.
Over each local coordinate chart $(U,z)$ on $C$,
$\cL$ has two distinguished trivializations corresponding to the
two square roots $\sqrt{\de z}$;
we choose one of these arbitrarily for each chart.
Then the transition map for $\cL$ between charts
$(U,z)$ and $(U',z')$ is
\begin{equation}
(z,s) \sim (z',s' = \alpha_{z,z'} s), \quad \mbox{where} \qquad \alpha_{z,z'} = \frac{\sqrt{\de z}}{\sqrt{\de z'}}.
\end{equation}

\begin{defn} \label{def:hitchin-component-SLN}
The \ti{Hitchin component}
is a set of stable $SL(N,\C)$-Higgs bundles
$(E, \bar\partial_E, \varphi_\bu)$, parameterized by $\bu \in \cB$,
as follows:
\begin{itemize}
 \item $E$ is the smooth vector bundle
 \begin{equation} \label{eq:E-SLN}
 	E = \cL^{N-1} \oplus \cL^{N-3} \oplus \cdots \oplus \cL^{-N+3} \oplus \cL^{-N+1}.
 \end{equation}
 Our distinguished local trivializations
 of $\cL$ induce distinguished local trivializations of $E$.
 Note that the exponents appearing in \eqref{eq:E-SLN}
 are the diagonal entries of the matrix $H$ from \eqref{eq:H-SLN}. Thus the
 transition maps between distinguished trivializations of $E$ are
\begin{equation} \label{eq:E0-transition-SLN}
	\alpha_{z,z'}^H = \begin{pmatrix} \alpha_{z,z'}^{N-1} & & & & \\ & \alpha_{z,z'}^{N-3} & & & \\ & & \ddots & & \\
      & & & \alpha_{z,z'}^{-N+3} & \\ & & & & \alpha_{z,z'}^{-N+1}
    \end{pmatrix} .
\end{equation}
 \item $\bar\partial_E$ is the holomorphic structure on $E$ induced from the
 one on $\cL$.
 \item Fix a chart $(U,z)$ and write $\phi_n = P_{n,z} \de z^n$.
 The Higgs field $\varphi_\bu \in \End E \otimes K_C$ is, relative
 to the distinguished local trivialization of $E$,
\begin{equation} \label{eq:higgs-field-SLN}
\varphi_{\bu,z} = \left(X_- + \sum_{n=1}^{N-1} P_{n+1,z} X_{n}\right) \de z.
\end{equation}
(Note that this indeed makes global sense, i.e.~$\alpha^H_{z,z'} \varphi_{\bu,z} \alpha^H_{z',z} = \varphi_{\bu,z'}$.)
\end{itemize}
\end{defn}

\begin{ex}
For $N=5$, (for one choice of normalizations of the $X_n$),
\begin{equation}
 \varphi_\bu = \begin{pmatrix} 0 & 2 P_2 &2 P_3 &P_4 &P_5 \\
2 & 0 & \sqrt{6} P_2 & \sqrt{6} P_3& P_4\\
0 & \sqrt{6}  & 0 & \sqrt{6} P_2 & 2P_3\\
0 & 0 & \sqrt{6}  & 0 & 2P_2\\
0 & 0 & 0 & 2 & 0
\end{pmatrix} \de z.
\end{equation}
Here and below, when working in a single coordinate
chart $(U,z)$, we sometimes drop the explicit
subscripts $z$ to reduce clutter.
Note that the characteristic polynomial of this matrix is
\begin{equation}
t^5 -20P_2 t^3 - 14 \sqrt{6} P_3 t^2 - (24P_4- 64 P_2^2) t -(24P_5-32\sqrt{6} P_2 P_3),
\end{equation}
so with our conventions, the $P_n$ are \ti{not} the coefficients of the characteristic polynomial, but
they both determine and can be recovered from these coefficients.
\end{ex}

When $N$ is even, the Hitchin component depends on the choice of spin structure.
When $N$ is odd, only even powers of $\cL$ appear, so in fact the Hitchin
component does not depend on the spin structure.

\subsection{The bilinear pairing}

The bundle $E$ given by \eqref{eq:E-SLN}
has a nondegenerate complex bilinear pairing $S$, i.e.~an isomorphism
$S: E \to E^*$, coming from the
fact that $\cL^{-n} = (\cL^n)^*$. In our distinguished trivializations
this is simply
\begin{equation}\label{eq:S}
S = \begin{pmatrix} & & & 1\\ & & 1 & \\
& \mbox{\reflectbox{$\ddots$}} & & \\ 1 & & & \end{pmatrix}.
\end{equation}
The antidiagonal symmetry of $X_{\pm}$ and the $X_n$ can be restated
as saying that they are self-adjoint with respect to $S$,
i.e.
\begin{equation}
S^{-1} X_n^T S = X_n
\end{equation}
and similarly for $X_\pm$. Thus, for any $\bu \in \cB$,
the Higgs field $\varphi_\bu$ of \eqref{eq:higgs-field-SLN} is also $S$-self-adjoint,
\begin{equation} \label{eq:higgs-S-self-adjoint}
S^{-1} \varphi_\bu^T S = \varphi_\bu.
\end{equation}
We define $\End_S E$ to be the subalgebra of traceless $S$-skew-adjoint
endomorphisms,
\begin{equation}
\End_S E = \{S: S^{-1} \chi^T S = -\chi \, \text{ and } \, \Tr S = 0\} \, \subset \, \End E.
\end{equation}
We then have
\begin{lem} \label{lem:vanishing-lemma-SLN}
If $\chi \in \End_S E$, then $[X_+, \chi] = 0$
if and only if $\chi = 0$.
\end{lem}
\begin{proof}
This follows directly from Proposition \ref{prop:Xn-facts-SLN}, which says that if $[X_+, \chi] = 0$, then
$\chi$ is a combination of $X_0,  X_1, \ldots, X_N$, and thus is $S$-self-adjoint.
\end{proof}

\subsection{The natural metric} \label{sec:natural-metric}

The Higgs bundle corresponding to the origin of $\cB$ is particularly
important:

\begin{defn}The \ti{uniformizing Higgs bundle} is
the element $(E,\bar\partial_E,\varphi_\bzero)$
of the Hitchin component, where
 $\bzero = (0,0,\dots,0) \in \cB.$
\end{defn}
Here is the reason for the name.
By the uniformization theorem, the conformal class
determined by the complex structure on $C$ contains a unique Riemannian
metric $g_\natural$ with constant curvature $-4$.
More generally, $g_\natural / R^2$ is the unique metric with
constant curvature $-4 R^2$.
This in turn induces a metric on $E$, as follows:
\begin{defn}
The \ti{natural metric} $h_{\natural}(R)$ on the bundle $E$ of \eqref{eq:E-SLN} is orthogonal with respect to the decomposition \eqref{eq:E-SLN},
and on $\cL^n \subset E$, is induced by $g_\natural / R^2$, i.e.,
\begin{equation}
 h_\natural(R) = R^n g_{\natural}^{-n/2}  \qquad \text{ on } \cL^n \subset E.
\end{equation}
We write $h_\natural$ for $h_\natural(R=1)$.
\end{defn}

Thus, viewing Hermitian metrics
as maps $E \to \overline{E}^*$, we have
\begin{equation} \label{eq:h-natural-R-dependence}
  h_\natural(R) = h_\natural \circ {R^H}.
\end{equation}

For future use we also describe $h_\natural(R)$ relative to the distinguished local trivializations of $E$.
In a local coordinate chart $(U,z)$, we can write
\begin{equation}
 g_\natural = \lambda_{\natural,z}^2 \de z \de \bar{z}, \quad \mbox{where} \quad
\partial_{\bar z} \partial_z \log \lambda_{\natural,z} - \lambda_{\natural,z}^2 = 0.
\label{eq:lambda-natural-de}
\end{equation}
Then
\begin{equation} \label{eq:h-natural-local}
h_{\natural,z}(R) = R^{H} \lambda_{\natural,z}^{-H} = \begin{pmatrix} R^{N-1} \lambda_{\natural,z}^{{1-N}} & & & & \\ & R^{N-3} \lambda_{\natural,z}^{{3-N}} & & & \\ & & \ddots & & \\
      & & & R^{3-N} \lambda_{\natural,z}^{{N-3}} & \\ & & & & R^{1-N} \lambda_{\natural,z}^{{N-1}}
    \end{pmatrix}.
\end{equation}

Note that $h_\natural(R)$ is compatible with $S$ in the sense that,
using $S$ to identify $E \simeq E^*$, the dual metric induced by
$h_\natural(R)$ is equal to $h_\natural(R)$ itself.
This is expressed concretely by the equation
\begin{equation} \label{eq:S-h-compatible}
  S^{-1} h_\natural(R)^T = h_\natural(R)^{-1} \bar{S}
\end{equation}
where both sides are maps $\bar{E} \to E$.
\eqref{eq:S-h-compatible}
is straightforward to check directly using \eqref{eq:S},
\eqref{eq:h-natural-local}.
It follows in particular that $S$ intertwines the Chern connections
on $E$ and $E^*$:
\begin{equation} \label{eq:S-cov-const}
  S^{-1} \circ (\bar\partial_E)^T \circ S = \bar\partial_E, \qquad S^{-1} \circ (\partial_E^{h_\natural})^T \circ S = \partial_E^{h_\natural}.
\end{equation}

The next proposition, from \cite{MR1174252}, explains the importance of $h_\natural(R)$ for our purposes:
\begin{prop} \label{prop:h-natural-harmonic}
The harmonic metric on the uniformizing Higgs bundle
$(E,\bar\partial_E,\varphi_\bzero)$ with parameter $R$
is $h_\natural(R)$.
\end{prop}

\begin{proof} We just compute directly in the distinguished trivializations:
\begin{align}
    F_{D_{h_\natural(R)}} + R^2 \left[\varphi_\bzero, \varphi_\bzero^{\dagger_{h_\natural(R)}} \right] &= \left( \partial_{\bar z} \partial_z \log (\lambda_\natural) H + \left[X_-, \lambda_\natural^2 X_+ \right]\right) \de z \wedge \de \bar{z} \\
    &= \left( \partial_{\bar z} \partial_z \log (\lambda_\natural) - \lambda_\natural^2 \right) H \, \de z \wedge \de \bar{z} \\
    &= 0.
\end{align}
\end{proof}

\subsection{{$SL(N,\C)$}-opers} \label{sec:opers-SLN}

We now recall the notion of \ti{$SL(N,\C)$-oper}:

\begin{defn} \label{def:oper-SLN}
An $SL(N,\C)$-oper on $C$ is a tuple $\left( E,\nabla , F_\bullet \right)$:
\begin{itemize}
\item A rank $N$ complex vector bundle $E$ over $C$, equipped with
a trivialization of the determinant bundle $\det E$,
\item A flat connection $\nabla$ on $E$,
\item A filtration $0 = F_0 \subset F_1 \subset \cdots \subset F_N = E$
of subbundles of $E$,
\end{itemize}
such that
\begin{itemize}
\item Each $F_n$ is holomorphic (with respect to the holomorphic
structure $\bar\partial_{\nabla}$),
\item If $\psi$ is a section of $F_n$ then $\nabla \psi$ lies in
the subbundle $F_{n+1} \otimes K_C \subset E \otimes K_C$,
\item The induced linear map
\begin{equation}
\bar \nabla: F_n / F_{n-1} \to F_{n+1} / F_n \otimes K_C
\end{equation}
is an isomorphism of line bundles, for $1 \le n \le N-1$.
\end{itemize}
\end{defn}

A flat holomorphic bundle $(E, \nabla)$ can admit at most one
filtration $F_\bullet$ satisfying the properties above.
Thus an $SL(N,\C)$-oper is a special sort of flat $SL(N,\C)$-bundle;
in fact, $SL(N,\C)$-opers form a holomorphic Lagrangian
subspace in the moduli space of flat $SL(N,\C)$-bundles.
However, we will not use this picture explicitly; our
constructions will produce the required filtrations directly in a local way.
For more background on opers see e.g.
\cite{MR2082709,Wentworth2014,dalakov-thesis}.

\subsection{A construction of {$SL(N,\C)$}-opers} \label{sec:oper-construction-SLN}

We now recall a construction of $SL(N,\C)$-opers which is
particularly convenient for our purposes. This construction has
its roots in the work of Drinfeld-Sokolov; see e.g. \cite{Zucchini:1994ik}
for a point of view close to ours.

We first describe a $1$-parameter family
of bundles $E_\hbar$ ($\hbar \in \C$), equipped
with holomorphic structures $\bar\partial_{E_\hbar}$
and holomorphic filtrations $F_{\hbar,\bullet}$.
Then for any $\bu \in \cB$ we will construct a corresponding $1$-parameter
family of connections $\nabla_{\hbar,\bu}$ ($\hbar \in \C^\times$),
compatible with the holomorphic structures and filtrations,
so that $(E_\hbar, \nabla_{\hbar,\bu}, F_{\hbar,\bullet})$ is a $1$-parameter
family of opers:

\begin{prop} \label{prop:oper-construction-SLN} We have the following:
\begin{itemize}
  \item For any $\hbar \in \C$,
the $SL(N,\C)$-valued transition functions
\begin{equation} \label{eq:Eh-transition-SLN}
	T_{\hbar,z,z'} = \alpha_{z,z'}^H \exp(\hbar \alpha_{z,z'}^{-1} \partial_z \alpha_{z,z'} X_+)
\end{equation}
define a holomorphic rank $N$ vector bundle
$(E_\hbar, \bar\partial_{E_\hbar})$ over $C$,
carrying a filtration $F_{\hbar,\bullet}$,
and equipped with a distinguished trivialization
for each local coordinate patch $(U,z)$ on $C$.

\item For any $\hbar \in \C^\times$ and $\bu \in \cB$,
there exists a canonical $SL(N,\C)$-oper
$(E_\hbar, \nabla_{\hbar,\bu}, F_{\hbar,\bullet})$,
compatible with the holomorphic structure $\bar\partial_{E_\hbar}$.
Relative to the distinguished trivializations of $E_\hbar$
on patches $(U,z)$ in the atlas
given by Fuchsian uniformization, $\nabla_{\hbar,\bu}$
is given by
\begin{equation} \label{eq:oper-local-SLN}
 \nabla_{\hbar,\bu,z} = \de + \hbar^{-1} \varphi_{\bu,z},
\end{equation}
where (as noted earlier)
\begin{equation} \label{eq:higgs-field-SLN-redux}
\varphi_{\bu,z} = \left(X_- + \sum_{n=1}^{N-1} P_{n+1,z} X_{n}\right) \de z.
\end{equation}

\end{itemize}
\end{prop}

The remainder of this section is devoted to the proof of
Proposition \ref{prop:oper-construction-SLN}.

When $\hbar = 0$, $E_0$ is just the
bundle $E$ described by \eqref{eq:E-SLN},
with transition functions $\alpha_{z,z'}^H$ as given in
\eqref{eq:E0-transition-SLN}.
The transition functions $T_{\hbar,z,z'}$ for $E_\hbar$ are a deformation of this.
However, there is still something to check:
\begin{lem} The transition functions \eqref{eq:Eh-transition-SLN}
obey the cocycle condition
\begin{equation} \label{eq:Eh-cocycle-SLN}
	T_{\hbar,z,z''} = T_{\hbar,z',z''} T_{\hbar,z,z'}.
\end{equation}
\end{lem}

\begin{proof} We will exhibit an alternative representation
\begin{equation} \label{eq:T-cocycle-alternative}
	T_{\hbar,z,z'} = M_{\hbar,z'} \alpha_{z,z'}^H M^{-1}_{\hbar,z},
\end{equation}
from which the cocycle condition \eqref{eq:Eh-cocycle-SLN} is immediate.

Fix some metric $g$ on $C$, represented locally as
\begin{equation}
	g = \lambda_z^2 \, \de z \de \bar{z},
\end{equation}
and let
\begin{equation}
	f_z = \partial_z \log \lambda_z.
\end{equation}
Then $\lambda_{z'} = \lambda_z \abs{\alpha_{z,z'}}^2$, whence
\begin{equation}\label{eq:transf-fz}
	f_{z'} = (\partial_z \alpha_{z,z'}) \alpha_{z,z'} + \alpha_{z,z'}^2 f_z.
\end{equation}

Now define
\begin{equation}
	M_{\hbar,z} = \exp(\hbar f_z X_+).
\end{equation}
Then we compute directly
\begin{align}
	M_{\hbar,z'} \alpha_{z,z'}^H M^{-1}_{\hbar,z} &= \exp(\hbar f_{z'} X_+) \alpha_{z,z'}^H \exp(-\hbar f_z X_+)  \\
	&= \alpha_{z,z'}^H \exp(\hbar f_{z'} \alpha_{z,z'}^{-2} X_+) \exp(-\hbar f_z X_+) \\
	&= \alpha_{z,z'}^H \exp(\hbar \alpha^{-1}_{z,z'} \partial_z \alpha_{z,z'} X_+) \\
	&= T_{\hbar,z,z'}
\end{align}
where the second equality uses the relation
\begin{equation} \label{eq:conjugation-Xplus}
 \alpha^{-H} X_+ \alpha^{H} = \alpha^{-2} X_+
\end{equation}
obtained by exponentiating \eqref{eq:SL2-relations}, and the third uses \eqref{eq:transf-fz}.
\end{proof}

We have now shown that the transition functions $T_{\hbar,z,z'}$
determine a vector bundle $E_\hbar$ over $C$. Moreover, the $T_{\hbar,z,z'}$
are holomorphic,
so $E_\hbar$ has a holomorphic structure $\bar\partial_{E_\hbar}$, represented by $\bar\partial$
in the distinguished local trivializations. (In other words, the distinguished local
trivializations are holomorphic.)  Note also that $T_{\hbar,z,z'}$ is an upper-triangular matrix, so it preserves
the filtration $F_{\hbar,\bullet}$, where $F_{\hbar,n}$ is spanned by the first $n$ basis vectors, and this filtration is defined globally.

Although $E_0 \not \simeq E_{\hbar}$ when $\hbar \neq 0$, all the other $E_\hbar$ are isomorphic:
\begin{prop}
For any $\lambda \in \C^\times$ and $\hbar \in \C$,
there is an isomorphism $E_\hbar \simarrow E_{\lambda^2 \hbar}$
given by $\lambda^H$ in the distinguished local trivializations.
\end{prop}

\begin{proof} We simply note that by \eqref{eq:conjugation-Xplus} and \eqref{eq:Eh-transition-SLN},
 $\lambda^H T_{z,z',\hbar} = T_{z,z',\lambda^2 \hbar} \lambda^{H}$.
\end{proof}

We now finally describe the connection $\nabla_{\hbar,\bu}$ on $E_\hbar$. For this purpose it is convenient to restrict the
choice of coordinate systems. We fix a complex projective structure on $C$,
i.e.~an atlas of coordinate charts $(U,z)$
with coordinates differing by M\"obius transformations,
\begin{equation} \label{eq:mobius}
	z' = \frac{az + b}{cz + d}, \qquad ad-bc=1.
\end{equation}
The particular complex projective structure we choose
is the one coming from Fuchsian uniformization, i.e.~the
realization of $C$ as a
quotient of the upper half-plane by a subgroup $\Gamma \subset SL(2,\R)$.
Now we can check:
\begin{lem} The formula \eqref{eq:oper-local-SLN} defines
a global connection in $E_\hbar$.
\end{lem}
\begin{proof} We must check that
\begin{equation}
	T_{\hbar,z,z'} \circ \nabla_{\hbar,\bu,z} \circ T_{\hbar,z,z'}^{-1} = \nabla_{\hbar,\bu,z'}
\end{equation}
when $z$ and $z'$ are related by \eqref{eq:mobius}.
We compute the LHS directly, writing $\alpha = \alpha_{z,z'}$ for simplicity. It is a sum of three terms. The first is
\begin{align}
& \hphantom{{}={}} \alpha^H \exp(\hbar \alpha^{-1} \partial_z \alpha X_+) \circ \de \circ \exp(-\hbar \alpha^{-1} \partial_z \alpha X_+) \alpha^{-H} \\
& = \de + \left[ \alpha^H \left((- \hbar \partial_z^2 \log \alpha) X_+\right)\alpha^{-H}  + (-\partial_z \log \alpha) H \right] \de z \\
& = \de + \left[(- \hbar \alpha^2 \partial_z^2 \log \alpha) X_+ + (-\partial_z \log \alpha) H\right] \de z.
\end{align}
Next is
\begin{align}
& \hphantom{{}={}} \hbar^{-1} \alpha^H \exp(\hbar \alpha^{-1} \partial_z \alpha X_+) X_- \exp(-\hbar \alpha^{-1} \partial_z \alpha X_+) \alpha^{-H} \\
& = \hbar^{-1} \alpha^H (X_- + (\hbar \partial_z \log \alpha) [X_+,X_-] + \half (\hbar \partial_z \log \alpha)^2 [X_+,[X_+,X_-]]) \alpha^{-H} \\
& =  \hbar^{-1} \alpha^H (X_- + (\hbar \partial_z \log \alpha) H - (\hbar \partial_z \log \alpha)^2 X_+) \alpha^{-H} \\
& =  \hbar^{-1} \alpha^{-2} X_- + (\partial_z \log \alpha) H - \hbar \alpha^2 (\partial_z \log \alpha)^2 X_+.
\end{align}
The transformation for the other terms in $\hbar^{-1} \varphi_{\bu,z}$ is simpler since they commute with $X_+$,
and we obtain
\begin{align}
& \hphantom{{}={}} \hbar^{-1} \alpha^H \exp(\hbar \alpha^{-1} \partial_z \alpha X_+) \left(\sum_{n=1}^{N-1} P_{n+1,z} X_{n}\right) \exp(-\hbar \alpha^{-1} \partial_z \alpha X_+) \alpha^{-H} \\
& = \hbar^{-1} \alpha^H \left(\sum_{n=1}^{N-1} P_{n+1,z} X_{n}\right) \alpha^{-H} \\
& = \hbar^{-1} \sum_{n=1}^{N-1} P_{n+1,z} \alpha^{2n} X_{n}.
\end{align}
Combining all these terms, the terms proportional to $H$ cancel nicely and
we get the desired result $\nabla_{\hbar,\bu,z'}$, \ti{except} for an extra term $\varepsilon X_+$, where
\begin{equation} \label{eq:error-term-SLN}
	\varepsilon = - \hbar ((\partial_z \alpha)^2 + \alpha^2 \partial_z^2 \log \alpha).
\end{equation}
It is precisely at this point where we have to use the restriction of the coordinate atlas. Indeed, for the transformations
\eqref{eq:mobius},
\begin{equation}
	\alpha = \pm (cz + d),
\end{equation}
and \eqref{eq:error-term-SLN} vanishes in this case.
\end{proof}

Finally note that the explicit formulas
\eqref{eq:oper-local-SLN} and \eqref{eq:Xminus-SLN} say
$\nabla_{\hbar,\bu,z}$ has nowhere-vanishing
entries on the first subdiagonal, and all entries
below this subdiagonal vanish. This is equivalent to
saying that $\nabla_{\hbar,\bu}$ obeys the
last condition in Definition \ref{def:oper-SLN}, and completes
the proof that $(E_\hbar, \nabla_{\hbar,\bu}, F_{\hbar,\bullet})$
is an $SL(N,\C)$-oper, thus completing the proof
of Proposition \ref{prop:oper-construction-SLN}.

\subsection{{$SL(N,\C)$}-opers and differential operators} \label{sec:SLN-opers-diff-ops}

This section is not used directly in the rest of the paper.
Its purpose is to recall the
sense in which $SL(N,\C)$-opers are equivalent to certain $N$-th
order linear scalar differential operators.
This construction is
well known; we include it here just to spell out its
relation with the description of opers in Proposition
\ref{prop:oper-construction-SLN} as connections on $E_{\hbar}$.

When $\psi$ is a holomorphic section of $\cL^{1-N}$,
let $\psi^{[N-1]}$ denote the $(N-1)$-jet of $\psi$,
which is a holomorphic section of the jet bundle $J^{N-1}(\cL^{1-N})$.
\begin{prop}
Fix $\bu \in \cB$.
There exists a unique holomorphic isomorphism
\begin{equation}
  \Phi_\bu: J^{N-1}(\cL^{1-N}) \simarrow E_\hbar
\end{equation}
such that:
\begin{itemize}
\item For any $\psi$, $\nabla_{\hbar,\bu} (\Phi_\bu(\psi^{[N-1]}))$ is valued
in the holomorphic line subbundle
$\cL^{N-1} \otimes K_C \simeq \cL^{N+1}$
of $E_\hbar \otimes K_C$,
\item The diagram
\smallskip
\begin{center}
\includegraphics{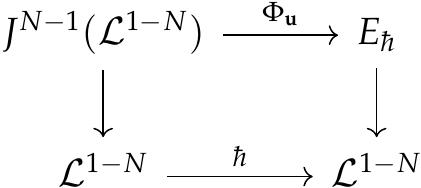}
\end{center}
\smallskip
is commutative, where the left vertical arrow is the projection to
the $0$-jet, and the right arrow is the quotient by $F_{N-1}$,
given in distinguished local trivializations by taking the
last component.
\end{itemize}
The map
\begin{equation}
  \cD_{\hbar,\bu}: \psi \mapsto \nabla_{\hbar,\bu} (\Phi_\bu(\psi^{[N-1]}))
\end{equation}
is a linear differential operator of order $N$, mapping
$\cL^{1-N} \to \cL^{N+1}$.
\end{prop}

This becomes much more concrete when we write
$\Phi_\bu$ relative to the distinguished
local trivializations in the Fuchsian atlas.
For instance, when $N=2$, we have
\begin{equation} \label{eq:Phi-SL2}
  \left[\partial_z + \hbar^{-1} \begin{pmatrix} 0 & P_2 \\ 1 & 0 \end{pmatrix}\right] \begin{pmatrix} - \hbar^2 \psi_z' \\ \hbar \psi_z \end{pmatrix} = \begin{pmatrix} -\hbar^2 \psi_z'' + P_2 \psi_z \\ 0 \end{pmatrix}.
\end{equation}
This equation implies that
\begin{equation}
  \Phi_\bu(\psi^{[1]}) = \begin{pmatrix} -\hbar^2 \psi_z' \\ \hbar \psi_z \end{pmatrix}.
\end{equation}
The $0$ in the bottom component of the RHS
of \eqref{eq:Phi-SL2} says $\nabla_{\hbar,\bu} (\Phi_\bu(\psi^{[1]}))$ is valued in the subbundle
$\cL \otimes K_C$; this condition determines $\Phi_\bu$ up to a constant
multiple which is fixed by requiring
that the bottom component of $\Phi_\bu(\psi^{[1]})$ is exactly $\hbar \psi_z$.
Thus we can read off from the top component of the RHS
of \eqref{eq:Phi-SL2} that $\cD_{\hbar,\bu}$ is represented locally by
\begin{equation}
  \cD_{\hbar,\bu,z} = -\hbar^2 \partial_z^2 + P_2.
\end{equation}

Similarly, for $N=3$, the analogue of \eqref{eq:Phi-SL2} is
\begin{equation}
  \left[\partial_z + \hbar^{-1} \begin{pmatrix} 0 & \sqrt{2} P_2 & P_3 \\ \sqrt{2} & 0 & \sqrt{2} P_2 \\ 0 & \sqrt{2} & 0 \end{pmatrix}\right] \begin{pmatrix} \frac{\hbar^3}{2} \psi_z'' - \hbar P_2 \psi_z \\ - \frac{\hbar^2}{\sqrt{2}} \psi_z' \\ \hbar \psi_z \end{pmatrix} = \begin{pmatrix} \frac{\hbar^3}{2} \psi_z''' - 2 \hbar P_2 \psi_z' - \hbar P'_2 \psi_z + P_3 \psi_z \\ 0 \\ 0 \end{pmatrix}
\end{equation}
which says that we have
\begin{equation} \label{eq:Phi-SL3}
  \Phi_\bu(\psi^{[2]}) = \begin{pmatrix} \frac{\hbar^3}{2} \psi_z'' - \hbar P_2 \psi_z \\ - \frac{\hbar^2}{\sqrt{2}} \psi_z' \\ \hbar \psi_z \end{pmatrix}
\end{equation}
and
\begin{equation} \label{eq:D-SL3}
  \cD_{\hbar,\bu,z} = \frac{\hbar^3}{2} \partial_z^3 - 2 \hbar P_2 \partial_z  - \hbar P'_2 + P_3.
\end{equation}
Note that $\Phi_\bu$ depends on $\bu$, through the $P_2$ in \eqref{eq:Phi-SL3}.
The $\bu$ dependence of $\cD_{\hbar,\bu}$ is thus more complicated
than one might naively guess: we already see that $P_2'$ appears in \eqref{eq:D-SL3},
despite the fact that only the $P_n$ and not their derivatives appear in the formula  \eqref{eq:oper-local-SLN} defining
$\nabla_{\hbar,\bu}$.

\section{The scaling limit, for {$G = SL(N,\C)$}} \label{sec:gaiotto-SLN}

\subsection{The main theorem for {$G = SL(N,\C)$}}

Fix an $SL(N,\C)$-Higgs bundle $(E, \bar\partial_E, \varphi)$ on $C$.  We then have the 2-parameter
family \eqref{eq:hitchin-family} of flat connections in $E$, depending on
$\zeta \in \C^\times$ and $R \in \R^+$, where we take
$h(R)$ to be the harmonic metric guaranteed by Theorem
\ref{thm:harmonic-existence}, and $D = D_{h(R)}$:
\begin{equation} \label{eq:two-param-family}
  \nabla_{R,\zeta} = \zeta^{-1} R \varphi + D_{h(R)} + \zeta R \varphi^{\dagger_{h(R)}}.
\end{equation}
We are going to consider the limits of certain 1-parameter subfamilies
of \eqref{eq:two-param-family},
obtained by taking $R \to 0$ and $\zeta \to 0$ simultaneously
while holding their ratio fixed.
In other words, fix some $\hbar \in \C^\times$ and set $\zeta = R \hbar$: then
\eqref{eq:two-param-family} becomes
\begin{equation} \label{eq:hitchin-family-rescaled}
  \nabla_{R,\hbar} = \hbar^{-1} \varphi + D_{h(R)} + \hbar R^2 \varphi^{\dagger_{h(R)}}.
\end{equation}
In \cite{Gaiotto2014}, Gaiotto proposed (and gave considerable
evidence for):

\begin{conj} \label{conj:gaiotto}
Suppose the $SL(N,\C)$-Higgs bundle $(E, \bar\partial_E, \varphi)$ is in the
Hitchin component, and fix some $\hbar \in \C^\times$.
Then as $R \to 0$ the connections $\nabla_{R,\hbar}$
converge to a connection $\nabla_{0,\hbar}$ in $E$, and there
exists a filtration $F_{\bullet}$
in $E$ such that $(E, \nabla_{0,\hbar}, F_{\bullet})$ is an $SL(N,\C)$-oper.
\end{conj}
We will prove the following explicit version
of Conjecture \ref{conj:gaiotto}:
\begin{thm} \label{thm:oper-limit-SLN}
Fix any $\bu \in \cB$.
Let $(E, \bar\partial_E, \varphi_\bu)$ be the corresponding Higgs
bundle in the Hitchin component, and let $h(R,\bu)$ be the family of
harmonic metrics on $E$ solving the rescaled Hitchin
equation \eqref{eq:harmonic-metric}. Let $F_\bullet$ be the filtration
\begin{equation}
F_n = \bigoplus_{i=1}^{n}  \cL^{N+1-2i} \subset E.
\end{equation}
Fix $\hbar \in \C^\times$ and let
\begin{equation} \label{eq:nabla-R-h-SLN}
  \nabla_{R,\hbar,\bu} = \hbar^{-1} \varphi_\bu + D_{h(R,\bu)} + \hbar R^2 \varphi_\bu^{\dagger_{h(R,\bu)}}.
\end{equation}
Then, as $R \to 0$ the flat connections $\nabla_{R,\hbar,\bu}$
converge to a flat connection
\begin{equation} \label{eq:limit-connection-SLN}
 \nabla_{0,\hbar,\bu} = \hbar^{-1} \varphi_\bu + D_{h_\natural} + \hbar \varphi_{\bzero}^{\dagger_{h_\natural}},
\end{equation}
and $(E, \nabla_{0,\hbar,\bu}, F_\bullet)$
is an $SL(N,\C)$-oper, equivalent to the $SL(N,\C)$-oper
$(E_\hbar, \nabla_{\hbar,\bu}, F_{\hbar,\bullet})$
of Proposition \ref{prop:oper-construction-SLN}.
\end{thm}

We emphasize that the harmonic metric $h(R,\bu)$ depends
on $R$, and indeed (as we will see) $h(R,\bu)$ diverges as $R \to 0$.
In particular,
we cannot simply drop the last term of
\eqref{eq:nabla-R-h-SLN} in the $R \to 0$ limit,
despite the explicit prefactor $R^2$; it
survives to become the last term of
\eqref{eq:limit-connection-SLN}, and is ultimately
responsible for the deformation of the holomorphic structure
as a function of $\hbar$.

\subsection{Proof of the main theorem for {$G = SL(2,\C)$}}
\label{sec:proof-SU2}

The case $N = 2$ of Theorem \ref{thm:oper-limit-SLN} is notationally simpler,
and contains the main ideas, so we do it separately.

Fix a coordinate patch $(U,z)$ on $C$,
and the corresponding distinguished trivialization of $E$.
Our first aim is to write an explicit local formula,
\eqref{eq:nabla-local-triv-SL2} below, for the
family of connections \eqref{eq:nabla-R-h-SLN} in $E$.

First we recall from \cite{MR1174252} that the decomposition
\begin{equation}
 E = \cL \oplus \cL^{-1}
\end{equation}
is orthogonal for $h(R,\bu)$.
Since $\cL^2 \simeq K_C$, $h(R,\bu)$ is
induced from a Hermitian metric $g(R,\bu)$ on $C$. In the local patch $(U,z)$,
\begin{equation} \label{eq:g}
g(R, \bu) = \lambda(R, \bu; z)^2 \de z \de \bar{z}
\end{equation}
for some positive real-valued function $\lambda(R,\bu;z)$ (which we
sometimes write $\lambda$ for short.) Then
\begin{equation}
h(R,\bu) = \begin{pmatrix} \lambda^{-1} & 0 \\ 0 & \lambda \end{pmatrix}.
\end{equation}

We now write the Chern connection $D_h$ explicitly.
Since the distinguished trivializations are holomorphic,
the $(0,1)$ part $\bar\partial_{D_h}$
is simply represented by $\bar\partial$. The $(1,0)$ part $\partial_{D_h}$
is then determined by unitarity with respect to $h$, which gives
\begin{equation}
\partial_{D_h} = \partial - \begin{pmatrix} \partial \log \lambda & 0 \\ 0 & -\partial \log \lambda \end{pmatrix},
\end{equation}
so altogether
\begin{equation} \label{eq:chern-explicit-SL2}
D_h = \de - \begin{pmatrix} \partial \log \lambda & 0 \\ 0 & -\partial \log \lambda \end{pmatrix}.
\end{equation}
Next, the choice of $\bu \in \cB$ just means
fixing a holomorphic quadratic differential $\phi_2 \in H^0(C, K_C^2)$.
Locally,
\begin{equation}
 \phi_2 = P_2 \, \de z^2
\end{equation}
where $P_2$ is a holomorphic function on $U$, and
\begin{equation} \label{eq:higgs-explicit-SL2}
  \varphi_\bu = \begin{pmatrix} 0 & P_2 \\ 1 & 0 \end{pmatrix} \de z,
  \qquad \varphi_\bu^{\dagger_h} = \begin{pmatrix} 0 & \lambda^{2}  \\ \lambda^{-2} \overline{P_2} & 0 \end{pmatrix} \de \bar{z}.
\end{equation}
Combining \eqref{eq:chern-explicit-SL2}, \eqref{eq:higgs-explicit-SL2}, \eqref{eq:nabla-R-h-SLN} gives the desired explicit representation,
\begin{equation} \label{eq:nabla-local-triv-SL2}
  \nabla_{R,\hbar,\bu} = \de + \hbar^{-1} \begin{pmatrix} 0 & P_2 \\ 1 & 0 \end{pmatrix} \de z -
  \begin{pmatrix} \partial \log \lambda & 0 \\ 0 & -\partial \log \lambda \end{pmatrix} +
  R^2 \hbar \begin{pmatrix} 0 & \lambda^{2}  \\ \lambda^{-2} \overline{P_2} & 0 \end{pmatrix} \de \bar{z}.
\end{equation}
Next we want to use \eqref{eq:nabla-local-triv-SL2} to understand
$\nabla_{R,\hbar,\bu}$ in the limit $R \to 0$. For this we need to
understand the behavior of $\lambda(R,\bu)$ as $R \to 0$.

Flatness of $\nabla_{R,\hbar,\bu}$ is equivalent to the fact that $h$ is the harmonic metric.
Thus computing the curvature of $\nabla_{R,\hbar,\bu}$ from \eqref{eq:nabla-local-triv-SL2} gives the harmonicity condition,
as an equation for $\lambda$:
\begin{equation} \label{eq:lambda-equation}
\partial_{\bar z} \partial_z \log \lambda - R^2(\lambda^{2}  - \lambda^{-2}\abs{P_2}^2 ) = 0.
\end{equation}

To get some intuition, first consider the special case $P_2 = 0$.
Then \eqref{eq:lambda-equation} specializes to
\begin{equation} \label{eq:lambda-equation-specialized}
  \partial_{\bar z} \partial_z \log \lambda - R^2 \lambda^{2} = 0,
\end{equation}
which says the metric $g(R, \bzero)$ of \eqref{eq:g} has constant curvature $-4 R^2$.
Thus $g(R, \bzero) = g_\natural / R^2$ where $g_\natural$ is the unique metric with constant curvature $-4$ (see \S\ref{sec:natural-metric}), and
$\lambda(R, \bzero) = \lambda_\natural / R$.

More generally when $P_2 \neq 0$, we use $g_\natural / R^2$ as background metric and write
\begin{equation}
g(R, \bu) = (g_\natural / R^2) \e^{2f}, \quad \lambda(R,\bu) = (\lambda_\natural / R) \e^f,
\end{equation}
where $f$ is a real-valued function on $C$. We claim that as $R \to 0$
\begin{equation} \label{eq:f-estimate}
\partial_z^a \partial_{\bar{z}}^b \, f = O(R^4)\ \ \mbox{for all}\ \  a,b \geq 0,
\end{equation}
(so in particular, $f = O(R^4)$).

To prove \eqref{eq:f-estimate}, first rewrite
\eqref{eq:lambda-equation} in terms of the Laplacian for $g_\natural$,
$\Delta_{g_\natural} =
\frac{4}{\lambda^2} \partial_{\bar z} \partial_z =
\frac{1}{\lambda^2} \left(\partial^2_{x} + \partial^2_{y} \right)$:
\begin{equation} \label{eq:Ng0}
N(f,R) = \Delta_{g_\natural} f + 4\left( 1 - \e^{2f} + R^4 \lvert P_2 \rvert^2 \e^{-2f}\right) = 0.
\end{equation}
The maximum principle shows  that for any $R \ge 0$, there exists at most one function $f$ such that $N(f,R) = 0$.
In fact, the method of upper and lower solutions shows that there is exactly one solution, but of course we
already know this when $R > 0$ from the existence and uniqueness of harmonic metrics, and $f = 0$ is
a solution when $R = 0$.

The linearization at $R=0$ is
\begin{equation} \label{eq:f-expansion}
\left. DN \right|_{(0,0)} (\dot{f}, 0) = \left( \Delta_{g_\natural} - 8\right) \dot{f},
\end{equation}
and this is an isomorphism as a map $\mathcal C^{k+2,\alpha}(C) \to \mathcal C^{k,\alpha}(C)$ for any $k \geq 0$ and $\alpha \in (0,1)$.
Note also that $N$ is a $\mathcal C^\infty$ mapping from a neighborhood of $0$ in $\mathcal C^{k+2,\alpha}(C) \times \mathbb R$
to $\mathcal C^{k,\alpha}(C)$.  The Banach space implicit function theorem now gives the existence of a $\mathcal C^\infty$
map $\Psi: [0, R_0) \to \mathcal C^{k+2,\alpha}(C)$ such that $N(\Psi(R), R) = 0$ for $0 \leq R  < R_0$, and $\Psi(0) = 0$.
From the uniqueness it follows that $\Psi$ is independent of $k$ and $\alpha$, so that $\Psi(R)$ is a $\cC^\infty$ function on $C$
for each $R \geq 0$, and in fact, $(z,R) \mapsto \Psi(R)(z)$ lies in $\cC^\infty(C \times [0,R_0))$.   We can say even more:
since all data in $N$ is real analytic, the real analytic version of the implicit function
theorem \cite{MR2977424} shows that $\Psi$ is real analytic in $R$.   Finally, by the uniqueness of harmonic metrics,
$\Psi(R)$ must agree with the desired $f$ when $R > 0$.

The upshot of the last paragraph is that we may expand $f$ in a Taylor series around $R = 0$,
\begin{equation}
	f = R f_1 + R^2 f_2 + \cdots
\end{equation}
Substituting this series into \eqref{eq:Ng0}, we see that
$f_1 = f_2 = f_3 = 0$, and hence we get \eqref{eq:f-estimate} as desired.

It follows that as $R \to 0$ we have
\begin{equation}
  \lambda = \frac{\lambda_\natural}{R} + O(R^3).
\end{equation}
Substituting this in \eqref{eq:nabla-local-triv-SL2}, we see that as $R \to 0$,
$\nabla_{R,\hbar,\bu}$ converges to
\begin{equation} \label{eq:nabla0-local-triv-SL2}
 \nabla_{0,\hbar,\bu} = \de + \hbar^{-1} \begin{pmatrix} 0 & P_2 \\ 1 & 0 \end{pmatrix} \de z - \begin{pmatrix}
\partial \log \lambda_\natural & 0 \\ 0 & - \partial \log \lambda_\natural \end{pmatrix} +
\hbar \begin{pmatrix} 0 & \lambda_\natural^2 \\
0 & 0 \end{pmatrix} \de \bar{z}.
\end{equation}
This is the desired \eqref{eq:limit-connection-SLN}.

\medskip

It is instructive to see directly that $(E,\nabla_{0,\hbar,\bu},F_\bullet)$
is an $SL(2,\C)$-oper, where $F_\bullet$ is the filtration
\begin{equation}
  0 \subset \cL \subset E.
\end{equation}
For this the key is the lower left entry $\hbar^{-1} \de z$
in \eqref{eq:nabla0-local-triv-SL2}, mapping
$\cL \to \cL^{-1} \otimes K_C \simeq \cL$.
The two salient facts about this are:
\begin{itemize}
\item Its $(0,1)$ part is trivial, so
$\cL$ is a holomorphic subbundle of $(E,\bar\partial_{\nabla_{0,\hbar,\bu}})$;
\item Its $(1,0)$ part is nowhere vanishing, i.e.,
$\bar \nabla_{0,\hbar,\bu}: \cL \to (E / \cL) \otimes K_C$ is an isomorphism of line bundles.
\end{itemize}
These conditions say precisely that
$(E, \nabla_{0,\hbar,\bu}, F_\bullet)$ is an $SL(2,\C)$-oper.

\medskip

Finally we show that
$(E, \nabla_{0,\hbar,\bu}, F_\bullet)$ is equivalent to the
$SL(2,\C)$-oper $(E_\hbar, \nabla_{\hbar,\bu}, F_{\hbar,\bullet})$
of Proposition \ref{prop:oper-construction-SLN}.
Comparing \eqref{eq:nabla0-local-triv-SL2}
to the desired form \eqref{eq:oper-local-SLN},
we see that we need to change our local trivializations
by a gauge transformation
which eliminates the last two terms in \eqref{eq:nabla0-local-triv-SL2},
i.e.~by a matrix of the form
\begin{equation}
  M_{\hbar,z} = \begin{pmatrix} 1 & \hbar \beta \\ 0 & 1 \end{pmatrix},
\end{equation}
where $\partial_{\bar z} \beta = \lambda_\natural^2$.
Because of the equation \eqref{eq:lambda-natural-de} for $\lambda_\natural$,
there is a natural candidate, $\beta = \partial_z \log \lambda_\natural$, leading to
\begin{equation} \label{eq:gauge-xform-SL2}
  M_{\hbar,z} = \begin{pmatrix} 1 & \hbar \partial_z \log \lambda_\natural \\ 0 & 1 \end{pmatrix}.
\end{equation}
Relative to the new local trivializations, the transition maps
from patches $(U,z)$ to $(U',z')$ become the ones we
wrote in \eqref{eq:T-cocycle-alternative}; these are
indeed the transition maps of $E_\hbar$.
What remains is to compute
$\nabla_{0,\hbar,\bu}$ in the new trivializations.
Computing directly $M_{\hbar,z} \circ \nabla_{0,\hbar,\bu} \circ M^{-1}_{\hbar,z}$ from \eqref{eq:nabla0-local-triv-SL2}, \eqref{eq:gauge-xform-SL2}
we obtain
\begin{equation} \label{eq:nabla0-new-triv-SL2}
  \de + \hbar^{-1} \begin{pmatrix} 0 & P_2 \\ 1 & 0
  \end{pmatrix} \de z +
  \hbar \begin{pmatrix} 0 & \frac{2 (\partial_z \lambda_\natural)^2 - \lambda_\natural \partial_z^2 \lambda_\natural}{\lambda_\natural^2} \\
  0 & 0
  \end{pmatrix} \de z.
\end{equation}
If our coordinate patch $(U,z)$ is in the atlas given
by Fuchsian uniformization, then the explicit form of the
hyperbolic metric in the upper half-plane gives
$\lambda_\natural = \frac{\I}{z - \bar{z}}$,
and then the last term in \eqref{eq:nabla0-new-triv-SL2}
vanishes. Thus \eqref{eq:nabla0-new-triv-SL2} reduces to the desired \eqref{eq:oper-local-SLN}.
This finishes the proof of Theorem \ref{thm:oper-limit-SLN}
in case $N=2$.

\subsection{Proof of the main theorem for {$G = SL(N,\C)$}}
\label{sec:proof-SUN}

Now we prove Theorem \ref{thm:oper-limit-SLN}
in full generality. The proof
is essentially the same as for $N=2$, with three differences:
\begin{itemize}
  \item The notation is less transparent, because
we cannot write everything in terms of explicit
$2 \times 2$ matrices.
\item The harmonic metrics $h(R,\bu)$ are no longer
determined by a single function on $C$, so we have to
study a coupled system instead of a
single scalar equation.
\item The harmonic metrics $h(R,\bu)$ may not be
diagonal in the distinguished trivializations.
\end{itemize}

As in the case $N=2$, the main technical issue is to control
the harmonic metric $h(R,\bu)$ in the limit $R \to 0$. We will show that
in this limit $h(R,\bu)$ approaches the natural metric
$h_\natural(R)$ of \S\ref{sec:natural-metric}. To say this precisely:
define $\End_{\natural,R}E$ to be the set of endomorphisms
which are self-adjoint with respect to $h_\natural(R)$.
Then $h(R,\bu): E \to \overline{E}^*$ can be written uniquely as
\begin{equation} \label{eq:h-chi-SLN}
  h(R,\bu) = h_\natural(R) \e^{\chi(R,\bu)} = \e^{\half \overline{\chi(R,\bu)}^T} h_\natural(R) \e^{\half \chi(R,\bu)}
\end{equation}
where $\chi(R,\bu) \in \End_{\natural,R} E$.
We will show that $\chi(R,\bu) \to 0$ as $R \to 0$:
\begin{lem} \label{lem:chi-estimate-SLN}
We have
\begin{equation}
  \chi(R,\bu) = O(R^4).
\end{equation}
\end{lem}

\begin{proof}
First some notation: for $\bu = (\phi_2, \dots, \phi_N) \in \cB$
and $\alpha \in \R^+$,
we let
\begin{equation}
 \alpha \bu = (\alpha^2 \phi_2, \dots, \alpha^N \phi_N) \in \cB.
\end{equation}
Now define
\begin{equation} \label{eq:chi-equation-SLN}
  N_\bu(\chi,R) = \left[\bar\partial_E, \e^{-\chi} \circ \partial_E^{h_\natural} \circ \e^\chi \right] + \left[\varphi_\bu, \e^{-\chi} \circ \varphi_{R^2 \bu}^{\dagger_{h_\natural}} \circ \e^\chi \right].
\end{equation}
We proceed in steps:
\begin{enumerate}
\item \label{item:op-type} For any fixed $R$, $N_\bu(\cdot, R)$ is a nonlinear operator
\begin{equation}
  N_\bu(\cdot, R): \Omega^0(\End_S E) \to \Omega^2(\End_S E).
\end{equation}
\item \label{item:N-harmonic} For any $R > 0$ and $\chi \in \End_{\natural,R} E$, we have
$N_\bu(\chi,R) = 0$ iff $h_\natural(R) \e^{\chi}$ is the harmonic metric for $\varphi_\bu$ with parameter $R$.
\item \label{item:N-zero} $N_\bu(0,0) = 0$.
\item \label{item:linearization-bijective} The linearization $\left. D_\chi N_\bu \right|_{(0,0)}$ is bijective.
\item \label{item:solution} There exists a real analytic $\chi(R,\bu) \in \End_{\natural,R} E \cap \End_S E$ for $R \in [0,R_0)$ such that
\begin{equation} \label{eq:N-solution-SLN}
 N_\bu(\chi(R,\bu),R) = 0.
\end{equation}
\item \label{item:taylor-R4} The first nonzero term in the Taylor expansion of
$\chi(R,\bu)$ around $R=0$ appears at order $R^4$.
\end{enumerate}

For (\ref{item:op-type}), what needs to be checked is that
$N_{\bu}(\cdot, R)$ preserves $\End_S E$.
This is a straightforward calculation using the compatibility
between $S$ and the rest of the data, as expressed in
\eqref{eq:S-cov-const}, \eqref{eq:S-h-compatible},
\eqref{eq:higgs-S-self-adjoint}.

For (\ref{item:N-harmonic}) we also compute directly.
The curvature of the Chern connection $D_h$
for the metric $h = h_\natural(R) \e^{\chi} = h_\natural {R^H} \e^\chi$ is
\begin{equation}
F_{D_h} = \left[\bar\partial_E, \partial_E^h \right] = \left[\bar\partial_E, \e^{-\chi} {R^{-H}} \circ \partial_E^{h_\natural} \circ {R^H} \e^\chi \right] =
\left[\bar\partial_E, \e^{-\chi} \circ \partial_E^{h_\natural} \circ \e^\chi \right],
\end{equation}
while
\begin{equation} \label{eq:h-adjoint-SLN}
  \varphi_\bu^{\dagger_h} = \e^{-\chi} {R^{-H}} \circ \varphi_\bu^{\dagger_{h_\natural}} \circ {R^H} \e^\chi = R^{-2} \e^{-\chi} \circ \varphi_{R^2 \bu}^{\dagger_{h_\natural}} \circ \e^\chi.
\end{equation}
Combining these gives
\begin{equation}
 F_{D_h} + R^2 [\varphi_\bu,\varphi_\bu^{\dagger_h}] = N_\bu(\chi,R)
\end{equation}
which is the desired result.

For (\ref{item:N-zero}), observe first that
\begin{equation}
  N_\bu(0,0) = \left[\bar\partial_E, \partial_E^{h_\natural}\right] + \left[\varphi_\bu, \varphi_\bzero^{\dagger_{h_\natural}} \right].
\end{equation}
This would vanish if $\varphi_\bu$ were replaced by $\varphi_\bzero$ since $h_\natural$ is the harmonic metric
for the Higgs field $\varphi_\bzero$,
by Proposition \ref{prop:h-natural-harmonic} (with $R=1$).
However, the difference $\varphi_\bu - \varphi_\bzero$ is a sum of terms $X_n$, all of which commute with
$\varphi_\bzero^{\dagger_{h_\natural}}$ since $\varphi_\bzero^{\dagger_{h_\natural}}$ is proportional to $X_+$. We conclude by recalling that
$[X_n, X_+] = 0$ by \eqref{eq:Xplus-commutator}.

For (\ref{item:linearization-bijective}) we compute that
\begin{equation}
\mathcal L_\bu (\dot\chi) :=  \left. D_\chi N_\bu\right|_{(0,0)} (\dot\chi) =
\bar\partial_E \partial_E^{h_\natural}\dot\chi + \left[\varphi_\bu, \left[\varphi_\bzero^{\dagger_{h_\natural}}, \dot\chi \right]\right].
\end{equation}
We wish to show this operator has trivial kernel. First consider the case $\bu = \bzero$.
Using the $L^2$ pairing induced by $h_\natural$, we have
\begin{equation}
\IP{\dot\chi, \mathcal L_\bzero\dot\chi} = \norm{\partial_E^{h_\natural} \dot\chi}^2 + \left\lVert \left[\varphi_\bzero^{\dagger_{h_\natural}}, \dot\chi \right] \right\rVert^2.
\end{equation}
By Lemma \ref{lem:vanishing-lemma-SLN}, the second term on the right is strictly positive if $\dot\chi \neq 0$,
so $\mathcal L_\bzero$ has trivial kernel. It can be deformed among elliptic operators to the self-adjoint operator
$\bar\partial_E \partial_E^{h_\natural}$, and hence has index zero, which means that $\mathcal L_\bzero$ is also surjective,
and hence an isomorphism $\cC^{k+2,\alpha} \to \cC^{k,\alpha}$ for any $k \geq 0$.

We now extend this to a statement about $\cL_\bu$.  To this end, we use the grading on $\End_S E$ where a matrix has grade $k$ if, in the distinguished
local trivializations, its nonzero entries are $k$ steps
above the diagonal.  Notice that $\mathcal L_\bzero$ preserves this grading.
Moreover, we have
\begin{equation}
\mathcal L_\bu (\cdot) - \mathcal L_\bzero(\cdot) = \left[\varphi_\bu - \varphi_\bzero, \left[\varphi_\bzero^{\dagger_{h_\natural}}, \cdot\right]\right]
\end{equation}
and this strictly increases the grading since
both $\varphi_\bu - \varphi_\bzero$
and $\varphi_\bzero^{\dagger_{h_\natural}}$ are strictly
upper triangular. It follows that $\mathcal L_\bu$ also has trivial kernel.
(Indeed, given operators $A,B$
where $A$ preserves a grading and $B$ increases it,
$(A+B)v = 0$ implies that $A$ annihilates the lowest-grade
component of $v$.) Then, by the same remarks as above,
$\cL_\bu$ is also surjective.

To obtain (\ref{item:solution}) we can apply the implicit function theorem,
exactly as in the case $N=2$ above, to deduce the existence
of a smooth function $\chi(R, \bu)$ such that $N_\bu( \chi(R, \bu), R) \equiv 0$ for $0 \leq R \leq R_0$.  As before,
this solution is real analytic in $R$ and $z$ jointly.
Moreover, writing $\dagger$ for $\dagger_{h_\natural(R)}$ we compute
directly from \eqref{eq:chi-equation-SLN}
\begin{equation}
N_\bu(\chi, R)^{\dagger} = \e^{\chi^{\dagger}} N_\bu(\chi^\dagger, R) \e^{-\chi^\dagger},
\end{equation}
so if $\chi$ is a solution then $\chi^\dagger$ is as well.
Thus the uniqueness in the implicit function theorem forces
$\chi = \chi^\dagger$, i.e.~$\chi(R,\bu) \in \End_{\natural,R} E$.

Finally, for (\ref{item:taylor-R4}) we simply plug the Taylor series
\begin{equation}
\chi(R,\bu) = \sum_{n=1}^\infty R^n \chi_n(\bu)
\end{equation}
into \eqref{eq:N-solution-SLN}, and expand in powers
of $R$. From \eqref{eq:chi-equation-SLN} we have
\begin{equation}
N_\bu(\chi,R) = N_{\bu}(\chi,0) + O(R^4).
\end{equation}
Thus at order $R^1$ we have to solve
\begin{equation}
\mathcal L_\bu (\chi_1) = 0,
\end{equation}
which we have already seen implies $\chi_1 = 0$.
Similarly we get $\chi_2 = \chi_3 = 0$.
This finishes the proof.
\end{proof}

Now we are ready to prove Theorem \ref{thm:oper-limit-SLN}.
We just substitute $h(R,\bu) = h_\natural(R) \e^{\chi(R,\bu)}$ in
\eqref{eq:nabla-R-h-SLN}, obtaining (using \eqref{eq:h-adjoint-SLN})
\begin{equation}
  \nabla_{R,\hbar,\bu} = \hbar^{-1} \varphi_\bu + \e^{-\chi(R,\bu)} \circ D_{h_\natural} \circ \e^{\chi(R,\bu)} + \hbar \e^{-\chi(R,\bu)} \circ \varphi_{R^2 \bu}^{\dagger_{h_\natural}} \circ \e^{\chi(R,\bu)}.
\end{equation}
In the limit $R \to 0$ we have $\chi(R,\bu) \to 0$, so this reduces to
\begin{equation}
  \nabla_{0,\hbar,\bu} = \hbar^{-1} \varphi_\bu + D_{h_\natural} + \hbar \varphi_{\bzero}^{\dagger_{h_\natural}},
\end{equation}
as desired.

To verify that $(E, \nabla_{0,\hbar,\bu}, F_\bullet)$
is equivalent to the oper of Proposition
\ref{prop:oper-construction-SLN}, we proceed just as we did
for $N=2$: introduce the matrix
\begin{equation}
  M_{\hbar,z} = \exp\left(\hbar (\partial_z \log \lambda_\natural) X_+\right)
\end{equation}
and compute directly
$M_{\hbar,z} \circ \nabla_{0,\hbar,\bu} \circ M_{\hbar,z}^{-1}$,
giving
\begin{equation}
  \de + \hbar^{-1} \varphi_\bu + \hbar \varepsilon X_+,
\end{equation}
where the ``error term'' $\varepsilon = \frac{2 (\partial_z \lambda_\natural)^2 - \lambda_\natural \partial_z^2 \lambda_\natural}{\lambda_\natural^2}$
vanishes when the
local coordinate $(U,z)$ is in the atlas given by
Fuchsian uniformization. This completes the proof.

\section{The scaling limit, for simple {$G$}}

In \S\ref{sec:gaiotto-SLN} we have stated and proved our main theorem for the
group $G = SL(N,\C)$. In this final section we generalize
to arbitrary simple, simply connected complex Lie groups $G$. This goes along the same lines
as for $G = SL(N,\C)$ --- indeed all the essential computations really have to do with
a principal $SL(2,\C)$ subgroup, which already appeared in the $G = SL(N,\C)$ case.
Thus we will be fairly brief.

We follow the conventions in \cite{Labourie2014}.

\subsection{Simple complex Lie algebras, Kostant subalgebra, and involutions} \label{sec:lie-stuff}

Fix a simple, simply connected
complex Lie group $G$ and a Cartan subgroup $H \subset G$.
Let $\Delta \subset \mathfrak{h}^*$ be the set of roots of $(G,H)$.
Choose a positive subset
$\Delta^+ \subset \Delta$, and let $\Pi \subset \Delta^+$ be the set of simple roots. Then there is a Chevalley basis for $\fg$:
\begin{equation}
 \fg = \left< \{h_{\alpha_i} \}_{\alpha_i \in \Pi}, \{ x_{\alpha} \}_{\alpha \in \Delta} \right>.
\end{equation}
The associated Borel subalgebra is
\begin{equation}
 \fb = \left< \{h_{\alpha_i} \}_{\alpha_i \in \Pi}, \{ x_{\alpha} \}_{\alpha \in \Delta^+} \right> \subset \fg.
\end{equation}
Let $B \subset G$ be the corresponding Borel subgroup.

Kostant's principal $\fsl(2,\C) \subset \fg$ is spanned by
$(H, X_+, X_-)$, defined in terms of the Chevalley basis
as:
\begin{equation} \label{eq:principal-sl2-generators}
H = \sum_{\alpha \in \Delta^+} h_\alpha = \sum_{\alpha_i \in \Pi} r_{\alpha_i} h_{\alpha_i}, \quad
X_\pm = \sum_{\alpha_i \in \Pi} \sqrt{r_{\alpha_i}} x_{\pm \alpha_i},
\end{equation}
and obeying
\begin{equation} \label{eq:SL2-relations-redux}
 [H, X_\pm] = \pm 2 X_\pm, \qquad [X_+, X_-]=H.
 \end{equation}
The first equation in \eqref{eq:principal-sl2-generators}
defines the coefficients $r_{\alpha_i} \in \Z^+$.
$\mathfrak{g}$ decomposes
as a sum of irreducible representations $\fv_1, \dots, \fv_r$
of the principal $\fsl(2,\C)$, where $r$ is the rank of $G$.
Each $\fv_n$ has odd complex dimension; we arrange them
in increasing order and write
\begin{equation}
 \dim_\C \fv_n = 2m_n+1.
\end{equation}
Then $(m_1, m_2, \dots, m_r)$ are the \ti{exponents} of $G$.

Let $X_{n}$ be a highest-weight vector in $\mathfrak{v}_n$,
i.e.~one obeying
\begin{equation}\label{eq:basiscommuteG}
[X_+, X_n] = 0, \qquad [H,X_n] = 2 m_n H.
\end{equation}
$X_n$ is determined up to a scalar multiple,
which we do not fix.

Let $\rho: \mathfrak{g} \rightarrow \mathfrak{g}$ be the
conjugate-linear Cartan involution associated
to the Chevalley system.  Its defining properties are that
it preserves $\mathfrak{h}$ (but is not the identity there) and satisfies
$\rho(x_{\alpha})= -x_{-\alpha}$.
$\rho$ determines a Hermitian form on $\mathfrak{g}$:
\begin{equation} \label{eq:hermitian-form-G}
 \langle Y_1, Y_2 \rangle = -B \left( Y_1, \rho(Y_2) \right)
\end{equation}
where $B$ is the Killing form on $\fg$.

Another important involution is given by:
\begin{prop}[\cite{MR1174252} Proposition 6.1]\label{lem:Gpositive}
There is a complex-linear involution
$\sigma: \mathfrak{g} \rightarrow \mathfrak{g}$ characterized by
\begin{equation}
 \sigma|_{\mathrm{ker}(\mathrm{ad}(X_+))} = -1, \qquad
\sigma(X_-) = -X_-.
\end{equation}
Moreover, $\sigma$ obeys
\begin{equation} \label{eq:sigma-rho-commute}
  \sigma \circ \rho=\rho \circ \sigma.
\end{equation}
\end{prop}

The fixed locus of $\rho$ is a real Lie subalgebra of $\fg$, corresponding to a compact real form $K$ of $G$. Similarly the fixed locus of $\lambda = \sigma \circ \rho$ is a real Lie subalgebra of $\fg$,
corresponding to a split real form $G^r$.

Both $\rho$ and $\sigma$
preserve the principal $\fsl(2,\C)$
(though they act non-trivially on it).
It follows that the principal embedding
$\iota: SL(2,\C) \hookrightarrow G$
restricts to $\iota: SL(2,\R) \hookrightarrow G^r$,
and to $\iota: SU(2) \hookrightarrow K$.

When $G = SL(N,\C)$, if we choose the standard Cartan subgroup
and simple roots, we have $\rho(X) = -X^\dagger$
and $\sigma(X) = -SX^TS^{-1}$. Thus in this case $K = SU(N)$.

\subsection{Higgs bundles}

Fix a compact Riemann surface $C$ of genus $g \ge 2$.
\begin{defn}
A \ti{$G$-Higgs bundle over $C$} is a tuple $(P, \bar\partial_P, \varphi)$:
\begin{itemize}
\item A principal $G$-bundle $P$ over $C$,
\item A holomorphic structure $\bar\partial_P$ on $\ad P$,
\item A holomorphic section $\varphi$ of $\ad P \otimes K_C$.
\end{itemize}
\end{defn}

\subsection{Harmonic reductions}

In the principal bundle setting, the analogue of a Hermitian metric $h$
is a reduction of structure group from $G$ to $K$, i.e.~a
principal $K$-bundle $Q \subset P$.
So, suppose given a $G$-Higgs bundle $(P, \bar\partial_P, \varphi)$ with
a reduction $Q$.
Since the involution $\rho$ of $\fg$ is $K$-invariant,
using the reduction $Q \subset P$ it
induces an involution $\rho_Q$ of $\ad P$
(which acts trivially on $\ad Q$.)

When $G = SL(N,\C)$, $K = SU(N)$,
and $P$ is the bundle of frames in $E$,
a reduction $Q \subset P$
is equivalent to a Hermitian metric $h$ in $E$ inducing the trivial
metric on $\det E$; namely $Q$ consists of all frames which are
unitary for $h$.
Then $\rho_Q(\varphi) = -\varphi^{\dagger_h}$.

There is a unique connection $D_Q$ in $\ad Q$ whose $(1,0)$ part is
$\bar\partial_P$ (Chern connection),
\begin{equation}
  D_Q = \bar\partial_P + \partial_P^Q,
\end{equation}
where
\begin{equation}
  \partial_P^Q = \rho_Q \circ \bar\partial_P \circ \rho_Q.
\end{equation}
We denote its curvature $F_{D_Q} \in \Omega^{2}(\ad Q)$.
Now we can formulate the analogue of the harmonic metrics
from \S\ref{sec:harmonic-metrics}:
\begin{defn}
Given a $G$-Higgs bundle $(P, \bar\partial_P, \varphi)$, and $R \in \R^+$,
a \ti{harmonic reduction with parameter $R$} is a
reduction of structure group to $K$, $Q \subset P$,
such that
\begin{equation} \label{eq:harmonic-reduction}
 F_{D_Q} - R^2 [\varphi, \rho_Q(\varphi)] = 0.
\end{equation}
\end{defn}

\subsection{Real twistor lines}

Given a $G$-Higgs bundle $(P, \bar\partial_P, \varphi)$ with harmonic
reduction $Q$ there is a corresponding family of flat connections
in $P$, given by the formula
\begin{equation} \label{eq:hitchin-family-G}
  \nabla = \zeta^{-1} R \varphi + D_Q - \zeta R \rho_Q(\varphi), \qquad \zeta \in \C^\times.
\end{equation}
Indeed, the statement that $\nabla$ is flat for all $\zeta \in \C^\times$
is equivalent to \eqref{eq:harmonic-reduction}.

\subsection{The Hitchin component} \label{sec:hitchin-component-G}

\begin{defn}
The \ti{Hitchin base} is the vector space
\begin{equation}
	 \cB = \bigoplus_{n=1}^r H^0(C, K_C^{m_n+1}).
\end{equation}
\end{defn}
We will denote points of $\cB$ by $\bu = (\phi_1, \dots, \phi_r)$,
where $\phi_n \in H^0(C, K_C^{m_n + 1})$.
(Warning: this is \ti{inconsistent} with our notation in
\S\ref{sec:background-SLN}: what we call $\phi_n$ here would
have been called $\phi_{n+1}$ there.)

As in \S\ref{sec:hitchin-component-SLN} we fix a spin structure
$\cL$ on $C$.
Let $P_0$ denote the principal $\C^\times$-bundle of
frames in $\cL$; it carries a canonical holomorphic structure
induced from the one in $\cL$.
The distinguished trivializations of $\cL$ associated to patches $(U,z)$
induce distinguished elements of $P_0$.
Then:

\begin{defn} \label{def:hitchin-component-G}
The \ti{Hitchin component}
is a set of $G$-Higgs bundles
$(P, \bar\partial_P, \varphi_\bu)$, parameterized by $\bu \in \cB$,
as follows:
\begin{itemize}
 \item The bundle $P$ is
\begin{equation} \label{eq:PG}
 P = P_0 \times_{\C^\times} G,
\end{equation}
where we embed $\C^\times \hookrightarrow G$
by $\alpha \mapsto \alpha^H$.

\item $\bar\partial_P$ is the holomorphic structure on $P$ induced from the
one on $P_0$.

\item The Higgs field $\varphi_\bu \in \ad P \otimes K_C$ is,
in the distinguished trivializations,
\begin{equation} \label{eq:higgs-field-G}
\varphi_{\bu,z} = \left(X_- + \sum_{n=1}^{r} P_{n,z} X_{n}\right) \de z.
\end{equation}

\end{itemize}
\end{defn}

The following crucial fact is proven in \cite{MR1174252}:
\begin{thm} \label{thm:harmonic-existence-G}
Given a $G$-Higgs bundle in the Hitchin component,
and any $R \in \R^+$,
there exists a unique harmonic reduction $Q$ with parameter $R$.
\end{thm}

\subsection{The linear involution}

The involution $\sigma: \fg \to \fg$ of \S\ref{sec:lie-stuff}
induces an automorphism $\sigma$ of $\ad P$,
where $P$ is the bundle \eqref{eq:PG}. (To see this involution is well
defined we use the fact that $\sigma(H) = H$.)
Note that
\begin{equation} \label{eq:sigma-higgs}
  \sigma(\varphi_\bu) = - \varphi_\bu.
\end{equation}
We define $\ad_\sigma P$ to be the
$\sigma$-invariant part of $\ad P$.

\subsection{The natural reduction}

Recall from \S\ref{sec:natural-metric} the natural
metric on $\cL$, induced by uniformization with parameter $R$.
Let $Q_{\natural,0}(R) \subset P_0$ be
the circle bundle of unit-norm vectors in $\cL$.
In the distinguished local trivializations of $P_0$,
$Q_{\natural,0}(R)$ is the $U(1)$-orbit of
$(\lambda_\natural / R)^\half$.

\begin{defn}
The \ti{natural reduction} $Q_{\natural}(R)$ of the bundle $P$
of \eqref{eq:PG} is
\begin{equation}
  Q_\natural(R) = Q_{\natural,0}(R) \times_{U(1)} K \subset P_0 \times_{\C^\times} G = P,
\end{equation}
where we use the
fact that for $\abs{\alpha} = 1$ we have $\alpha^H \in K$.
We abbreviate $Q_\natural(R = 1)$ as $Q_\natural$.
In particular, (cf. \eqref{eq:h-natural-R-dependence})
\begin{equation}
  Q_\natural(R) =  R^{-H/2} Q_\natural.
\end{equation}
\end{defn}

\begin{prop} \label{prop:Q-natural-harmonic-G}
The harmonic reduction on the Higgs bundle
$(P, \bar\partial_P, \varphi_\bzero)$ with parameter $R$
is $Q_\natural(R)$.
\end{prop}
\begin{proof} This is essentially the same computation
as in Proposition \ref{prop:h-natural-harmonic}. The
curvature $F_{D_{Q_\natural(R)}}$ is induced from
$F_{D_{Q_{\natural,0}(R)}} = \partial \bar\partial \log \lambda_\natural$
through the embedding $U(1) \hookrightarrow K$, so
\begin{equation}
 F_{D_{Q_\natural(R)}} = (\partial \bar\partial \log \lambda_\natural) H .
\end{equation}
In the distinguished local trivializations of $P$,
$Q_\natural(R)$ is the $K$-orbit
$(\lambda_\natural / R)^{H/2} K$.
Thus, in the distinguished local trivializations
$\rho_{Q_\natural(R)}$ acts by
$(\lambda_\natural / R)^{H/2} \circ \rho \circ (\lambda_\natural / R)^{-H/2}$.
So we have
\begin{align}
  R^2 \left[\varphi_\bzero, \rho_{Q_\natural(R)}(\varphi_\bzero)\right] &= R^2 \left[X_- \de z, (\lambda_\natural / R)^{H/2} \rho((\lambda_\natural / R)^{-H/2} X_-) \de \bar{z} \right] \\
  &= - \lambda_\natural^2 [X_-, X_+] \de z \wedge \de \bar{z} \\
  &= \lambda_\natural^2 H \, \de z \wedge \de \bar{z}.
\end{align}
Combining these and using \eqref{eq:lambda-natural-de} gives the result.
\end{proof}

\subsection{{$G$}-opers} Now we recall the notion of
$G$-oper. The principal bundle version of the filtration
from \S\ref{sec:opers-SLN} is a reduction to the Borel
subgroup $B \subset G$.

\begin{defn} \label{def:opers-G} A \ti{$G$-oper} on $C$ is a tuple
$(P, \nabla, F)$ where:
\begin{itemize}
\item $P$ is a principal $G$-bundle,
\item $\nabla$ is a flat connection in $P$,
\item $F \subset P$ is a reduction to $B \subset G$,
\end{itemize}
such that
\begin{itemize}
\item $F$ is holomorphic, with respect to the holomorphic structure on $P$
induced by $\nabla$,
\item $\nabla$ is in ``good position'' with respect to $F$, in the following
sense. Choose locally a connection $\nabla_B$ on $P$ induced from a flat
holomorphic connection on $F$, and consider $\nabla - \nabla_B \in
\Omega^{1,0}(\ad P)$. Changing the choice of $\nabla_B$ leaves invariant the
class $[\nabla - \nabla_B] \in \Omega^{1,0}(\ad P / \! \ad F)$. Now
$(\ad P / \! \ad F) = F \times_B (\fg/\fb)$, from which it follows that
the set of $B$-orbits of $(\ad P / \! \ad F)$
is just the set of $B$-orbits of $\fg / \fb$. Thus, at each point of $C$,
$\nabla$ determines a $B$-orbit $\cO_\nabla$ of
 $\fg / \fb \otimes \Omega^{1,0}$.
The ``good position'' condition is that $\cO_\nabla$ is the orbit containing
$[X_-] \de z$ (note this condition is independent of the choice of local
coordinate, since the $B$-orbit of $[X_-]$ contains all scalar multiples
of $[X_-]$.)
\end{itemize}
\end{defn}

\subsection{A construction of {$G$}-opers}
Next we recall a construction of $G$-opers parallel
to \S\ref{sec:oper-construction-SLN}.

\begin{prop} \label{prop:oper-construction-G} We have the following:
\begin{itemize}
  \item For any $\hbar \in \C$,
the transition functions
\begin{equation} \label{eq:Eh-transition-G}
  T_{\hbar,z,z'} = \alpha_{z,z'}^H \exp\left(\hbar \alpha_{z,z'}^{-1} \partial_z \alpha_{z,z'} X_+\right)
\end{equation}
define a holomorphic $G$-bundle
$(P_{\hbar}, \bar\partial_\hbar)$ over $C$,
with a reduction to a $B$-bundle $F_\hbar$,
and equipped with a distinguished trivialization
for each local coordinate patch $(U,z)$ on $C$.

\item For any $\hbar \in \C^\times$ and $\bu \in \cB$,
there exists a canonical $G$-oper
$(P_{\hbar}, \nabla_{\hbar,\bu}, F_\hbar)$,
compatible with the holomorphic structure $\bar\partial_{\hbar}$.
Relative to the distinguished trivializations of $P_{\hbar}$
on patches $(U,z)$ in the atlas
given by Fuchsian uniformization, $\nabla_{\hbar,\bu}$
is given by
\begin{equation} \label{eq:oper-local-G}
 \nabla_{\hbar,\bu,z} = \de + \hbar^{-1} \varphi_{\bu,z},
\end{equation}
where (as we have stated before)
\begin{equation} \label{eq:higgs-field-G-redux}
\varphi_{\bu,z} = \left(X_- + \sum_{n=1}^{r} P_{n,z} X_{n}\right) \de z.
\end{equation}

\end{itemize}
\end{prop}

\begin{proof}
The checks that $P_\hbar$ and its connection $\nabla_{\hbar,\bu}$
are globally well defined are
just as in the proof of Proposition \ref{prop:oper-construction-SLN}:
indeed the computations there only involved the principal
$\fsl(2,\C)$ generators $H, X_+, X_-$
and their commutation relations with the $X_n$, so they
go through unchanged.

The $B$-reduction
$F_\hbar$ is represented by $B \subset G$ in the
distinguished trivializations; this is indeed globally defined,
since $X_+, H \in \fb$ and thus the transition
functions $T_{\hbar,z,z'}$ of \eqref{eq:Eh-transition-G} are valued in $B$.
To see that $\nabla_{\hbar,\bu}$ is an oper, we can compute in any
distinguished trivialization, and locally take $\nabla_B$ to be
the trivial connection; then by \eqref{eq:oper-local-G},
the class $[\nabla_{\hbar,\bu}-\nabla_B]$ is
represented by the projection of $\hbar^{-1} \varphi_{\bu,z}$ to $\fg / \fb$.
Using \eqref{eq:higgs-field-G-redux} and the fact that each $X_n \in \fb$,
this projection is simply $\hbar^{-1} [X_-] \de z$. This is in the $B$-orbit
of $[X_-] \de z$, as desired.
\end{proof}

\subsection{The main theorem for general {$G$}}

Now we state and prove our main theorem for general
complex simple simply connected $G$:

\begin{thm} \label{thm:oper-limit-G}
Fix any $\bu \in \cB$.
Let $(P, \bar\partial_P, \varphi_\bu)$ be the corresponding Higgs
bundle in the Hitchin component, and let $Q(R,\bu) \subset P$
be the family of harmonic reductions solving the rescaled Hitchin
equation \eqref{eq:harmonic-reduction}. Let $F$ be the
reduction of $P$ to $B$ which, in each distinguished local
trivialization, is given by $B \subset G$.

Fix $\hbar \in \C^\times$ and let
\begin{equation} \label{eq:nabla-R-h-G}
  \nabla_{R,\hbar,\bu} = \hbar^{-1} \varphi_\bu + D_{Q(R,\bu)} - \hbar R^2 \rho_{Q(R,\bu)}(\varphi_\bu).
\end{equation}
Then, as $R \to 0$ the flat connections $\nabla_{R,\hbar,\bu}$
converge to a flat connection
\begin{equation} \label{eq:limit-connection-G}
 \nabla_{0,\hbar,\bu} = \hbar^{-1} \varphi_\bu + D_{Q_\natural} - \hbar \rho_{Q_\natural}(\varphi_{\bzero}),
\end{equation}
and $(P, \nabla_{0,\hbar,\bu}, F)$
is a $G$-oper, equivalent to the $G$-oper
$(P_{\hbar}, \nabla_{\hbar,\bu}, F_{\hbar})$
of Proposition \ref{prop:oper-construction-G}.
\end{thm}

In parallel to the proof of Theorem
\ref{thm:oper-limit-SLN}, the main technical issue is to show that
the harmonic reduction $Q(R,\bu)$ approaches $Q_\natural(R)$
as $R \to 0$. We can formulate this as follows.
Let $\Ad_{\natural,R} P$ be the group of automorphisms of $P$
preserving $Q_\natural(R)$, and $\ad_{\natural,R} P$ its
Lie algebra. We have
\begin{equation} \label{eq:Q-chi-G}
  Q(R,\bu) = \e^{-\half \chi(R,\bu)} Q_\natural(R)
\end{equation}
for a unique $\chi(R,\bu) \in \ad_{\natural,R} P$.
Then, parallel to Lemma \ref{lem:chi-estimate-SLN}:
\begin{lem} We have
\begin{equation}
  \chi(R,\bu) = O(R^4).
\end{equation}
\end{lem}

\begin{proof}
First some notation: for $\bu = (\phi_i)_{i=1}^r \in \cB$
and $\alpha \in \R^+$,
we let
\begin{equation}
 \alpha \bu = (\alpha^{m_i+1} \phi_i)_{i=1}^r \in \cB.
\end{equation}
Now define
\begin{equation} \label{eq:chi-equation-G}
  N_\bu(\chi,R) = \left[\bar\partial_P, \e^{-\chi} \circ \partial_P^{Q_\natural} \circ \e^{\chi} \right] - \left[\varphi_\bu, \left(\e^{-\chi} \circ \rho_{Q_\natural} \right) (\varphi_{R^2 \bu}) \right].
\end{equation}
Then, just as in the proof of Lemma \ref{lem:chi-estimate-SLN},
we proceed in steps:
\medskip
\begin{enumerate}
\item \label{item:op-type-G} For any fixed $R$, $N_\bu(\cdot, R)$ is a nonlinear operator
\begin{equation}
  N_\bu(\cdot, R): \Omega^0(\ad_\sigma P) \to \Omega^2(\ad_\sigma P).
\end{equation}
\item \label{item:N-harmonic-G} For any $R > 0$ and $\chi \in \ad_{\natural,R} P$, we have
$N_\bu(\chi,R) = 0$ iff $\e^{-\half \chi} Q_\natural(R)$ is the harmonic reduction for $\varphi_\bu$ with parameter $R$.
\item \label{item:N-zero-G} $N_\bu(0,0) = 0$.
\item \label{item:linearization-bijective-G} The linearization $\left. D_\chi N_\bu \right|_{(0,0)}$ is bijective.
\item \label{item:solution-G} There exists a real analytic $\chi(R,\bu) \in \ad_{\natural,R} P \cap \ad_\sigma P$ for $R \in [0,R_0)$ such that
\begin{equation} \label{eq:N-solution-G}
 N_\bu(\chi(R,\bu),R) = 0.
\end{equation}
\item \label{item:taylor-R4-G} The first nonzero term
in the Taylor expansion of
$\chi(R,\bu)$ around $R=0$ appears at order $R^4$.
\end{enumerate}
\medskip
Each step is strictly parallel to the analogous step in
the proof of Lemma \ref{lem:chi-estimate-SLN}; we just mention
the necessary substitutions.
In step (\ref{item:op-type-G}) the necessary compatibility with $\sigma$
is \eqref{eq:sigma-rho-commute} and \eqref{eq:sigma-higgs}.
For step (\ref{item:N-harmonic-G}) we first note
that if $\chi \in \ad_{\natural,R} P$ then
\begin{equation}
  \rho_Q = \e^{-\half \chi} R^{-H/2} \circ \rho_{Q_\natural} \circ R^{H/2} \e^{\half \chi} = \e^{-\chi} R^{-H} \circ \rho_{Q_\natural}.
\end{equation}
Thus the curvature of $D_Q$ is
\begin{equation}
  F_{D_Q} = \left[\bar\partial_P, \partial_P^{Q} \right] =
  \left[\bar\partial_P, \e^{-\chi} {R^{-H}} \circ \partial_P^{Q_\natural} \circ {R^H} \e^{\chi} \right] = \left[\bar\partial_P, \e^{-\chi} \circ \partial_P^{Q_\natural} \circ \e^{\chi} \right],
\end{equation}
and
\begin{equation}
  \rho_Q (\varphi_\bu) = (\e^{-\chi} \circ \rho_{Q_\natural} \circ R^H)(\varphi_\bu) = R^{-2} (\e^{-\chi} \circ \rho_{Q_\natural})(\varphi_{R^2 \bu})
\end{equation}
where we used
\begin{equation}
  R^H \varphi_\bu = R^{-2} \varphi_{R^2\bu}.
\end{equation}
In step (\ref{item:N-zero-G}) we use
Proposition \ref{prop:Q-natural-harmonic-G}.
In step (\ref{item:linearization-bijective-G}) we compute
\begin{equation}
\mathcal L_\bu (\dot\chi) :=  \left. D_\chi N_\bu\right|_{(0,0)} (\dot\chi) =
\bar\partial_P \partial_P^{Q_\natural} \dot\chi - \left[\varphi_\bu,\left[\rho_{Q_\natural}(\varphi_\bzero), \dot\chi\right]\right],
\end{equation}
and then proceed as in Lemma \ref{lem:chi-estimate-SLN},
using the $L^2$ norm on $\ad P$
induced by $Q_\natural$ and the Hermitian pairing
\eqref{eq:hermitian-form-G}, and the grading on $\fg$ induced by $H$.
The remaining steps (\ref{item:solution-G}), (\ref{item:taylor-R4-G})
are just as in the proof of Lemma \ref{lem:chi-estimate-SLN}.
\end{proof}

The remainder of the proof of Theorem \ref{thm:oper-limit-G}
is also strictly parallel to that of Theorem \ref{thm:oper-limit-SLN},
so we omit it.

\bibliographystyle{utphys}
\bibliography{oper-limit}

\end{document}